\documentclass{amsart}
\usepackage[utf8]{inputenc}
\usepackage{amsthm}
\usepackage{amsmath}
\usepackage{tikz}
\usepackage{comment}
\usetikzlibrary{arrows}

\usepackage{enumerate}   
\usepackage{amssymb}
\usepackage{fancyhdr} 
\usepackage[left=3cm,right=3cm,top=3cm,bottom=3cm]{geometry}
\usepackage{xcolor}
\newcommand{\norm}[1]{\left\lVert#1\right\rVert}
\usepackage{graphicx}
\usepackage{ esint }
\newtheorem{theorem}{Theorem}[section]

\newtheorem{corollary}[theorem]{Corollary}
\newtheorem{lemma}[theorem]{Lemma}
\newtheorem{definition}[theorem]{Definition}

\theoremstyle{remark}

\newcommand{\veps}{\varepsilon}
\newcommand{\CC}{\mathbb C}
\newcommand{\RR}{\mathbb R}
\newcommand{\DD}{\mathbb D}

\newcommand{\TT}{\mathbb T}
\newcommand{\NN}{\mathbb N}

\newcommand{\ovdd}{\overline{\DD}^2}
\newcommand{\ovdt}{\overline{\DD}^3}
\newcommand{\ovD}{\overline{\DD}^d}
\newcommand{\m}[1]{\vert #1 \vert}
\newcommand{\TTI}{\TT^{\m{I}}}
\newcommand{\tk}{\theta_k}
\newcommand{\rk}{\rho_k}

\newcommand{\vb}[1]{V_{\beta_#1}}
\newcommand{\sed}{S(\eta,\ovd)}
\newcommand{\wed}{W(\eta,\ovd)}
\newcommand{\eio}{e^{i\theta}}
\newcommand{\dss}{\displaystyle\sum}
\newcommand{\dsi}{\displaystyle\int}
\newcommand{\Ab}{A_\beta^2(\DD^3 )}
\newcommand{\tu}{\theta_1}
\newcommand{\td}{\theta_2}
\newcommand{\tr}{\theta_3}
\newcommand{\I}[1]{\lbrace 1,\dots,#1\rbrace}
\newcommand{\Si}{S_I(\eta,\ovd)}
\newcommand{\Abd}{A^2_\beta(\DD^d)}
\newcommand{\Abdi}[1]{A^2_{\beta_{#1}}(\DD^d)}
\newcommand{\wik}{\omega_{I,k}(\ovd)}
\newcommand{\dsp}{\displaystyle\prod}
\newcommand{\oii}{(0,1)^{\m{I}}}

\def\Imm{\mathrm{Im}}
\def\Ree{\mathrm{Re}}
\def\ovd{\overline{\delta}}

\usepackage{url}
\usepackage[hidelinks]{hyperref}
\hypersetup{
	colorlinks=true,
	linkcolor= cyan,
	filecolor=green,
	citecolor =blue,      
	urlcolor=purple,
}

\author{Anne Dorval}
\address{Laboratoire de Math\'ematiques Blaise Pascal UMR 6620 CNRS, Universit\'e Clermont Auvergne, Campus universitaire des C\'ezeaux, 3 place Vasarely, 63178 Aubi\`ere Cedex, France.}
\email{anne.dorval@uca.fr}

\thanks{The author is partially supported by the grant ANR-24-CE40-0892-01 of the French National Research Agency ANR}

\begin{document}

\title{Compactness of composition operator on weighted Bergman spaces of the polydisc}

\begin{abstract}
We study composition operators induced by a smooth symbol between weighted Bergman spaces of the polydisc. We first prove a compactness criterion that only requires knowing what happens on the distinguished boundary. Then we prove simple geometric characterizations of boundedness and compactness on some $\Abd$,  particularly for $\beta > d-3$.
 
\end{abstract}
\maketitle
\section{Introduction}
Let $\mathcal U $ be a domain in  $\CC^d$, let $X$ be a Banach space of holomorphic functions on $\mathcal{U}$ and let $\phi $ be a holomorphic self map of $\mathcal{U}$ (we will denote  $ \phi \in \mathcal{O}(\mathcal{U},\mathcal{U})) .$ Then the composition operator with symbol $\phi$ is defined by $C_\phi(f) = f \circ \phi, \ f \in X. $ The first question to study is about boundedness: for which symbols $\phi $ does $C_\phi$ induced a bounded operator on $X$ ? When $X$ is the Hardy space or a weighted Bergman space  of the unit disc, Littlewood's Subordination principle states that every analytic function on the disc defines a bounded operator. However, as the dimension grows the problem gets trickier. On weighted Bergman spaces of the bidisc a characterisation using only first order derivatives for smooth symbols up to the boundary was given in \cite{Ko22}. It gets a lot harder to completely characterise  boundedness on Bergman spaces of the tridisc (see for example \cite{Baytridisc}, \cite{BD26}  or \cite{Ko22}). Other authors have worked with affine symbols (see \cite{BAYPOLY},\cite{SZ06a} or \cite{SZ06b}) or rational inner functions (see \cite{Bes25a} \cite{Bes25b},\cite{Bes26} or \cite{Bes26b}) for example.
The next question to study is the one about compactness. On the Hardy space of the unit disc, compactness of composition operator was characterised by Shapiro in \cite{Shap} using the Nevanlinna counting function of $\phi.$ On the Hardy space of the bidisc or the tridisc, the question also gets trickier:  in \cite{Baytridisc}, Bayart only gives  a necessary and a sufficient condition for compactness for smooth symbols up to the boundary using properties on the derivative of $\phi$. In \cite{CCS25}, Clos, Čučković and Şahutoğlu give a geometric compactness criterion on the unweighted Bergman space of the bidisc for smooth symbols that are Lipshitz on the boundary, precisely: let $\phi \in \mathcal{O}(\DD^2,\DD^2)  $ be Lipschitz on the closure, then $C_\phi : A^2(\DD^2) \rightarrow A^2(\DD^2)$ is compact if and only if $\phi(\ovdd) \cap \TT^2 = \emptyset$ and $ \phi(\ovdd \setminus \TT^2) \cap b\DD^2 = \emptyset,$ where $b\DD^2$ denote the boundary of the bidisc. \\
In this paper, we investigate compactness of composition operators on weighted Bergman spaces of the polydisc of smooth symbol that are of class $\mathcal{C}^1$ up to the boundary. Our first result, Theorem \ref{result_gen_comp}, is a technical result that characterise  the compactness using only the values of $\phi $ on the polydisc. As a consequence, we obtain a stability result for compatness on weighted Bergman spaces. \\
In the second part, we take a more geometric approach to characterise the  boundedness (resp compactness) of composition operators on $\Abd$ when $\beta \geq d-3$ (resp $\beta >d-3$). As for the bidisc, our characterisation of boundedness only uses first order derivatives and the behaviour of $\phi$ near the distinguished boundary of $\DD^d$. For $d\geq 1$ and $ j\in \I{d}, $ we define  the set $T_{j,d}$ by $ \lbrace z \in \ovD : \exists I\subset\I{d},\ \m{I} = j,\  \m{z_k} = 1, \ \forall k \in I  \rbrace  $ and our result is: 
\begin{theorem}
Let $d \geq 2. $ Let $\beta \geq d-3$. Let $\phi \in \mathcal{O}( \DD^d,\DD^d) \cap \mathcal{C}^1(\ovD). $ Then, $C_\phi : A^2_{\beta}(\DD^d) \rightarrow A^2_{\beta}(\DD^d)$ is bounded if and only if 
\begin{enumerate}[(a)]
    \item For all $ j = 1,\dots,d-1,\ \phi(\ovD \setminus T_{j+1,d} ) \cap T_{j+1,d} = \emptyset$
    \item For all  $I \subset \I{d}, \ \m{I } = j$ and all $ \xi \in T_{j,d} \setminus T_{j+1,d}$ such that $ \phi_I(\xi) \in \TT^j $ then $ \left( \dfrac{\partial\phi_i(\xi)}{\partial z_k }\right)_{i \in I, \ k \in P_I} $is invertible. 
\end{enumerate}
\end{theorem}
From this result, we deduce a compactness criterion which only requires knowing the behaviour of $\phi$ on the polycircle. It generalize Clos, Čučković and Şahutoğlu result to all dimensions.

\begin{theorem}
Let $d \geq 2$. Let $\phi \in \mathcal{O}(\DD^d ,\DD^d) \cap \mathcal{C}^1(\ovD). $ Let $\beta >d-3.$ Then 
$C_\phi : \Abd \rightarrow \Abd$ \text{ is compact  if and only if } $\forall j \in \I{d},\ \phi(\ovD\setminus T_{j+1,d} ) \cap T_{j,d} = \emptyset. $ 

\end{theorem}
Finally, in the last part, we provide several examples illustrating  our results. 

\section{Some general results}
\subsection{Notations}Let us start by introducing some notations that will be used throughout this paper. The unit vector $(1,\dots,1)$ will be denoted by $e$. 
For $\theta=(\theta_1,\dots,\theta_d)\in\RR^d,$ we will write $e^{i\theta}$ for $(e^{i\theta_1},\dots,e^{i\theta_d}).$ Any $z\in\overline{\DD}^d$ such that $z_k\neq 0$ for all $k$ will be uniquely written $((1-\rho_k)e^{i\theta_k})$
with $\theta_k\in[-\pi,\pi)$ and $\rho_k\in[0,1).$ 

For  $f : \TT^d \rightarrow A$  and $\theta \in \RR^d$, we write $f(\theta)$ for $f(e^{i\tu},\dots,e^{i\theta_d})$. 
If $f$ maps $A$ into $\CC^d,$ the map $f_I$, where $I=\{i_1,\dots,i_p\}\subset \{1,\dots,d\}$, will denote $(f_{i_1},\dots,f_{i_p}):A\to\CC^p.$ 
For $f,g : A \rightarrow \RR$, we shall write ${f\lesssim g} $ provided there exists $C>0$ such that for all $x \in A,\ f(x) \leq Cg(x).$ When we write $f\simeq g,$ we mean $f\lesssim g$ and $g\lesssim f$.

The Lebesgue measure on $\RR^d$ will be denoted by  $\lambda_d$ and the normalized surface measure on $\TT^d$ will be denoted by $\sigma_d$. Let $dA$ denote the normalized area measure on $\DD$. For $\beta >-1$, we put $dA_\beta(z) = (\beta+1)(1-\m{z}^2)^\beta dA(z)$. On the polydisc $\DD^d$, we define 
$$dV_\beta(z) = dA_\beta(z_1)\cdots dA_\beta(z_d), \hspace{5mm} z=(z_1,\dots,z_d) \in \DD^d.$$
For $\beta>-1,$ the Bergman space $\Abd$ is the function space defined by 
$$ \Abd = \left\lbrace f \in \mathcal{O}(\DD^d,\DD^d):\ \norm{f}^2_2 = \int_{\DD^d}\m{f(z)}^2dV_\beta(z)<+\infty \right\rbrace,$$ and the Hardy space $H^2(\DD^d)$ is defined by 
$$ H^2(\DD^d) = \left\lbrace f \in \mathcal{O}(\DD^d,\DD^d):\  \norm{f}^2_2 = \sup\limits_{0<r<1}\int_{\TT^d}\m{f(re^{i\theta})}^2d\sigma_d(z) <+\infty \right\rbrace.$$ 
For simplicity, we will sometimes write $A^2_{-1}(\DD^d)$ for $H^2(\DD^d).$ 

For $\eta \in \TT^d$ and $\ovd \in (0,2]^d$, the Carleson box is the set $${\sed = \lbrace z \in \DD^d :\ \m{z_j-\eta_j}<\delta_j,\ j=1,\dots,d\rbrace}.$$ If $I\subset \{1,\dots,d\}$, $S_I(\eta,\overline\delta)$
will mean $\lbrace z \in \DD^d :\ \m{z_j-\eta_j}<\delta_j,\ j\in I\rbrace$. \\
For $\eta = (e^{i\tu},\dots,e^{i\theta_d}) \in \TT^d$ and $\ovd \in (0,2]^d$, the Carleson window is the set $${\wed = \lbrace  (r_ke^{it_k})_{k=1,\dots,d} \in \DD^d :\ 1-\delta_k \leq r_k \leq 1, \ \m{\theta_k-t_k} \leq \delta_k, \ k=1,\dots,d\rbrace}.$$
Since  there exists $c>0$ such that, for all $\eta \in \TT^d $ and all $\ovd \in (0,2]^d,$ 
$$ \sed \subset \wed \subset S(\eta,c\ovd),$$ 
in some results that we will state, we may replace $\sed$ by $\wed,$ and conversely. 

\subsection{Characterisation of compactness }
First,  we characterise the compactness of $C_\phi : \Abdi{1} \rightarrow \Abdi{2}$ with a condition on the measure using Carleson windows as it was done for the continuity using Carleson boxes. It is what the two following theorems are all about. The condition varies depending on whether $\Abdi{1}$ is the Hardy space or not. The following theorem is a precise formulation of Theorem $2.5$ of \cite{Jaf91} :   

\begin{theorem}\label{caract_compacite_mesure}
\begin{enumerate}[(a)]
    \item Let $\beta>-1$ and $I_\beta: A^2_\beta(\DD^d) \rightarrow L^2(\mu) $ be the identity map. Then $I_\beta$ is a compact operator if and only if $$\forall \veps>0, \hspace{2mm} \exists \gamma >0,\hspace{2mm} \forall \ovd \in (0,2]^d,\hspace{2mm} \forall \eta \in \TT^d,\hspace{2mm} \min\limits_{i = 1,\dots,d}\delta_i \leq\gamma\hspace{2mm} \Rightarrow \hspace{2mm}\dfrac{\mu(W(\eta,\ovd))}{\prod\limits_{j=1}^d \delta_j^{2+\beta}} \leq \veps . $$ 
    
    \item Let $I_\beta : H^2(\DD^d) \rightarrow L^2(\mu)$ be the identity map. If $I_\beta$ is a compact operator then
$$\forall \veps>0, \hspace{2mm} \exists \gamma >0,\hspace{2mm} \forall \ovd \in (0,2]^d,\hspace{2mm} \forall \eta \in \TT^d,\hspace{2mm} \min\limits_{i = 1,\dots,d}\delta_i \leq\gamma\hspace{2mm} \Rightarrow \hspace{2mm}\dfrac{\mu(W(\eta,\ovd))}{\prod\limits_{j=1}^d \delta_j }\leq \veps . $$ 

\end{enumerate}

\end{theorem}

However, as the reciprocal does not hold on the Hardy space, we need one more result to prove compactness on Hardy spaces:

\begin{theorem}\label{comp_Hardy_Bergman}
Let $\beta_1 \geq-1 $ and $\phi \in \mathcal{O}(\DD^d,\DD^d).$ Assume that there exist $a >0$ and $ \veps : (0,a] \rightarrow (0,+\infty) $ with $\lim\limits_{0^+}\veps = 0, $ such that for all $\ovd \in (0,a]^d, $ all $\beta \in (-1,0] $ and all $\eta \in \TT^d, $
$$ \vb{1}(\phi^{-1}(\sed)) \leq \veps(\min(\delta_i, \ i=1,\dots,d))\dsp_{j=1}^d\delta_j^{2+\beta}. $$
Then $C_\phi : H^2(\DD^d) \rightarrow \Abdi{1}$  is compact. 
\end{theorem}

\begin{proof}
The idea will be to write $C_\phi$ as the limit of  finite rank operators (this proof will use the same ideas as the proof of Theorem  $8.2$ of \cite{Baytridisc}). Let us consider $\beta_1' = \beta_1+ (1+\beta) $ be such that $\beta_1' \rightarrow \beta_1 $ when $\beta \rightarrow -1. $ We aim to show that  there exist a sequence of finite rank operators $(T_N)$ and a sequence of positive real numbers $(\omega_N)_N$ going to zero when $N$ tends to $+\infty$ such that  $\norm{C_\phi-T_N }_{\mathcal{L}(\Abd, A^2_{\beta_1'}(\DD^d)) } \leq \omega_N$ for all $\beta\in (-1,0]$ with $\omega_N$ independent from $\beta. $ For simplicity reasons, we are only going to prove the case $d=2.$ Let $\gamma \in (0,1)$ and $N \in \NN$. For $f = \dss_{\alpha \in \NN} a_\alpha z_1^{\alpha_1}z_2^{\alpha_2}, $ we consider  $P_N(f) = \dss_{\m{\alpha} \leq N }a_\alpha z_1^{\alpha_1}z_2^{\alpha_2} $ of rank $N^2. $ We put $T_N = C_\phi \circ P_N$ and assume that $f$ is polynomial. We write $\vb{1',\phi}$ the image by $\phi$ of $\vb{1'} $ and $ \mu_{\gamma,\phi, \beta_1'}$ its restriction to $\DD^2 \setminus (1-\gamma)\DD^2. $ 
 We can write 
 \begin{align*}
     \norm{C_\phi(f) - T_N(f)}^2_{\beta_1'} & = \dsi_{\DD^2}\m{(f-P_N(f))\circ \phi}^2d\vb{1'} \\
     & = \dsi_{\DD^2}\m{f-P_N(f)}^2d\vb{1',\phi} \\
     & = \dsi_{(1-\gamma)\DD^2}\m{f-P_N(f)}^2d\vb{1',\phi} +\dsi_{\DD^2 \setminus (1-\gamma)\DD^2}\m{f-P_N(f)}^2d\mu_{\gamma,\phi, \beta_1'}.
 \end{align*}

Let us start by the second term. We want to apply Lemma $8.1$ of \cite{Baytridisc} to $ \mu_{\gamma,\phi, \beta_1'}. $ Let $\eta \in \TT^2 $ and $\ovd\in (0,2]^2$. If $\delta_1 \leq \gamma $ or $\delta_2 \leq \gamma$, then 
$$ \mu_{\gamma,\phi, \beta_1'}(\wed) \leq \vb{1'}(\wed) \leq \veps(\gamma) (\delta_1\delta_2)^{2+\beta}$$
with $\veps(\gamma) \xrightarrow [ \gamma \to 0]{} 0$ by assumptions. If $\delta_1 >\gamma$ and $\delta_2 > \gamma$, we cover $\wed \cap (\DD^2 \setminus(1-\gamma)\DD^2)$ as follows. Let $p_j = \lfloor \frac{2\delta_j}{\gamma} \rfloor$ for $j=1,2  $ (see the proof of Theorem $3.12$ of \cite{CMc95}). To simplify, we will write $p$ for $p_1+p_2 . $  Let $(\tk)_{k=1,\dots,p} $ be such that 
\begin{align*}
&\tu = -\delta_1, \hspace{3mm} \td = -\delta_1+\gamma, \hspace{3mm} \dots\ , \hspace{3mm} \theta_{p_1} = -\delta_1+(p_1-1)\gamma \\
&\theta_{p_1+1} = -\delta_2, \hspace{3mm} \theta_{p_1+2} = -\delta_2+\gamma, \hspace{3mm} \dots \ , \hspace{3mm} \theta_{p_1+p_2} = -\delta_2+(p_2-1)\gamma 
\end{align*}
and for
\begin{align*}
& k = 1,\dots, p_1, \hspace{3mm} \eta(k) = (e^{i\theta_k},1) \hspace{3mm} \text{and} \hspace{3mm} \ovd(k) = (\gamma, \delta_2) \\
& k = p_1+1 ,\dots, p_1+p_2 , \hspace{3mm} \eta(k) = (1,e^{i\theta_k}) \hspace{3mm} \text{and} \hspace{3mm} \ovd(k) = (\delta_1,\gamma).
\end{align*}
Then, $\wed \cap N \subset \bigcup\limits_{k=1}^pW(\eta(k),\ovd(k))$. Indeed, let $z\in \wed \cap N,\ z =(r_1e^{it_1},r_2e^{it_2})$. If $r_1 \geq 1-\gamma, $ there exists $k \in \lbrace 1,\dots, p_1 \rbrace$ such that $t_1 \in (\theta_k,\theta_k+\gamma)$ and if $r_2 \geq 1-\gamma$, there exists $k \in \{p_1+1,\dots, p_1+p_2\} $ such that $ t_2 \in (\theta_k,\theta_k+\gamma)$. Thus, $z \in W(\eta(k),\ovd(k))$ and $\wed \cap (\DD^2 \setminus(1-\gamma)\DD^2) \subset \bigcup\limits_{k=1}^{p_1+p_2}W(\eta(k),\ovd(k)). $ Hence, 
\begin{align*}
 \mu_{\gamma,\phi, \beta_1'}(\wed) &=    \mu_{\gamma,\phi, \beta_1'}(\wed\cap\DD^2 \setminus(1-\gamma)\DD^2 ) \\
 & \leq \dss_{k=1}^{p_1+p_2} \mu_{\gamma,\phi, \beta_1'}(W(\eta(k),\ovd(k)))\\
 & \leq \veps(\gamma)p_1\gamma^{2+\beta}\delta_2^{2+\beta} + \veps(\gamma) p_2\gamma^{2+\beta}\delta_1^{2+\beta}\\
 &\leq \veps(\gamma)\left(\dfrac{2\delta_1}{\gamma}\gamma^{2+\beta}\delta_2^{2+\beta}+ \dfrac{2\delta_2}{\gamma}\gamma^{2+\beta}\delta_1^{2+\beta}\right)\\
 &\leq 2\veps(\gamma)\delta_1\delta_2(\delta_1^{1+\beta}+\delta_2^{1+\beta})\\
 & \leq 4\veps(\gamma)(\delta_1\delta_2)^{2+\beta}. 
\end{align*}
By Lemma $8.1 $ of \cite{Baytridisc}, there exist $C>0$ independent of $\beta \in (-1,0], \ \gamma \in (0,1)$ and $N \in \NN$ such that for all $f$ holomorphic  polynomial
$$\dsi_{\DD^2}\m{f-P_N(f)}^2d\mu_{\gamma,\phi, \beta_1'} \leq \dsi_{\DD^2\setminus(1-\gamma)\DD^2}\m{(f-P_N(f))\circ \phi }^2d\vb{1'}  \leq C\veps(\gamma)^2\norm{f-P_N(f)}^2_2 \leq C\veps(\gamma)^2 \norm{f}_{2}^2. $$
Now, about the first term of the sum. We decompose $f-P_N(f)$ as 
$$f-P_N(f) = \dss_{k,l \geq N}a_{k,l}z_1^kz_2^l + \dss_{k <N \atop l\geq N}a_{k,l}z_1^kz_2^l+ \dss_{k \geq N \atop l < N}a_{k,l}z_1^kz_2^l = z_1^Nz_2^Ns_1(z)+z_2^Ns_2(z)+z_1^Ns_3(z). $$
Hence 
\begin{align*}
    \dsi_{(1-\gamma)\DD^2}\m{f-P_N(f)}^2d\vb{1',\phi} &\leq (1-\gamma)^{2N}\dsi_{\DD^2} \m{s_1}^2d\vb{1',\phi} +(1-\gamma)^{N}\dsi_{\DD^2} \m{s_2}^2d\vb{1',\phi} +(1-\gamma)^{N}\dsi_{\DD^2} \m{s_3}^2d\vb{1',\phi}\\
    & \leq C'((1-\gamma)^{2N}\norm{s_1}^2_{\beta_1'} +(1-\gamma)^{N}\norm{s_2}^2_{\beta_1'} +(1-\gamma)^{N}\norm{s_3}^2_{\beta_1'} ),
\end{align*}
with $C' = \sup\limits_{\beta \in (-1,0]}  \norm{C_\phi}_{\mathcal L(\Abd,A^2_{\beta_1'}(\DD^d))}<+\infty$. Indeed, our assumptions imply that we may apply Proposition $9.3$ of \cite{BAYPOLY}.
\begin{align*}
\norm{s_1}^2_{\beta_1'} &= \dss_{k,l\geq N} \dfrac{\m{a_{k,l}}^2}{(k-N+1)^{\beta_1'+1}(l-N+1)^{\beta_1'+1}}\\
&=  \dss_{k,l\geq N} \dfrac{\m{a_{k,l}}^2}{(k-N+1)^{\beta+1 + \beta_1+1}(l-N+1)^{\beta+1 +\beta_1+1}} \\
&= \dss_{k,l\geq N} \dfrac{\m{a_{k,l}}^2}{(k+1)^{\beta+1 }(l+1)^{\beta+1 }} \cdot \dfrac{((k+1)(l+1))^{\beta+1 }}{((k-N+1)(l-N+1))^{1+\beta+\beta_1+1}}\\
&\leq \dss_{k,l\geq N} \dfrac{\m{a_{k,l}}^2}{(k+1)^{\beta+1 }(l+1)^{\beta+1 }} \cdot \left(\dfrac{(k+1)(l+1)}{(k-N+1)(l-N+1)}\right)^{\beta+1}\\
& \leq (N+1)^{2(\beta+1)}\norm{f}^2_{2}.
\end{align*}  
In the same way, we can show that 
$$\norm{s_2}^2_{\beta_1'} ,\norm{s_3}^2_{\beta_1'}  \leq (N+1)^{2(\beta+1)}\norm{f}^2_{2}. $$
Hence, 
\begin{align*}
 \dsi_{(1-\gamma)\DD^2}\m{f-P_N(f)}^2d\vb{1',\phi} & \leq C'\norm{f}_2^2((1-\gamma)^{2N}(N+1)^{2(\beta+1)} + 2 (1-\gamma)^{N}(N+1)^{2(\beta+1)}) \\
 &\leq C'\norm{f}_2^2((1-\gamma)^{2N}(N+1)^{2} + 2 (1-\gamma)^{N}(N+1)^2), 
\end{align*}
because $\beta \in (-1,0]. $ Finally, 
$$ \norm{ C_\phi(f) -T_N(f)}^2_{\beta_1'} \leq C''(\veps(\gamma)^2 +(1-\gamma)^{2N}(N+1)^{2} + 2 (1-\gamma)^{N}(N+1)^2) \norm{f}_2^2,  $$
with $C'' >0$ depending only on $C$ and $C'$. We set $\gamma_N = \dfrac{1}{\log(N+1)}$, with $\veps(\gamma_N) \xrightarrow [N\to +\infty] {} 0,$ and 
$$ \omega_N =  C''(\veps(\gamma_N)^2 +(1-\gamma_N)^{2N}(N+1)^{2} + 2 (1-\gamma_N)^{N}(N+1)).$$ 
Thus,  $\omega_N \xrightarrow[N \rightarrow +\infty]{}0$ and for all holomorphic polynomials $f$ and all $\beta \in(-1,0], $we have  $\norm{C_\phi(f)-T_N(f)}_{\beta_1'}^2 \leq \omega_N\norm{f}^2_2.$ By letting  $ \beta $ go to $-1$, we get $\norm{C_\phi-T_N}_{\mathcal{L}(H^2(\DD^d), \Abdi{1})}^2 \leq \omega_N. $ Thus, $C_\phi : H²(\DD^d) \rightarrow \Abdi{1}$ is compact. 
\end{proof}

\subsection{Some technical lemmas }
Before stating our result, we need a lemma to estimate the volume of a certain subset which is going to be very useful in what follows (see \cite{BD26} for the proofs of all the lemmas): 
\begin{lemma}\label{equiv}
Let $\beta >-1$. There exist $C_1,C_2>0$ such that, for all measurable subsets $A$ of $\mathbb{R}^d$ and for all $\ovd \in (0,1)^d,$
$$C_1 (\delta_1\cdots \delta_d)^{(1+\beta)}\lambda_d(A) \leq V_\beta(\lbrace ((1-\rk)e^{i\tk})_{k=1,\dots,d}: 0\leq \rho_k\leq \delta_k, \theta \in A \rbrace) \leq C_2(\delta_1\cdots \delta_d)^{(1+\beta)}\lambda_d(A).$$
\end{lemma} 

Our result gives a characterisation of compactness of $C_\phi : \Abdi{1} \rightarrow \Abdi{2} $ using only the boundary values of $\phi $ provided $\phi$ is regular enough. But before stating it, we need some definitions: 

\begin{definition}
\begin{enumerate}[(a)]
\item Let $\varphi \in \mathcal{O}(\DD^d,\DD) \cap \mathcal{C}^1(\ovD)$. Assume that there exists $\theta \in \RR^d$ such that $\varphi(\theta) \in \TT. $ We define
$$P_\varphi = \left\lbrace k \in \I{d} : \dfrac{\partial\varphi(\theta)}{\partial z_k} \neq 0 \right\rbrace.$$
\item Let $\phi\in\mathcal O(\DD^d,\DD^d)\cap\mathcal C^1(\ovD)$ and 
let $I\subset\{1,\dots,d\}$. We say that $\phi_I$ touches the polycircle if $\phi_I(\TT^d)\cap\TT^{|I|}\neq\varnothing.$
\item Let $\phi\in\mathcal O(\DD^d,\DD^d)\cap\mathcal C^1(\ovD)$ and 
$I\subset\{1,\dots,d\}$ be such that $\phi_I$ touches the boundary. For $j\in I,$ we simply write $P_j=P_{\varphi_j}$ and we set $P_I=\bigcup\limits_{j\in I  } P_j.  $
\item Let $\phi\in\mathcal O(\DD^d,\DD^d)\cap\mathcal C^1(\ovD)$, let 
$I\subset\{1,\dots,d\}$ be such that $\phi_I$ touches the boundary, let  $\ovd \in (0,1)^{|I|}$ and let $k\in\{1,\dots,d\}$. We set 
$$ \omega_{I,k}(\ovd) = \left\lbrace\begin{array}{ll} 
1 & \text{provided } k \notin P_I \\
\min\lbrace \delta_j,\ j\in I, \ k\in P_j \rbrace      & \text{otherwise.}  \\
\end{array}\right.$$
\end{enumerate}
 \end{definition}

This number $\wik$ can be quite difficult to compute as it really depends on the sequence $(\delta_1,\dots,\delta_d). $ However, it is very useful to compare the sets $\phi^{-1}(\sed)$ with sets whose volumes we know how to compute: 

\begin{lemma}[See Lemma $2.8$ of \cite{BD26}]\label{lem:doubleinclusion}
Let $\phi\in\mathcal O(\DD^d,\DD^d)\cap \mathcal C^1(\ovD).$
\begin{enumerate}[(a)]
\item For all $I\subset\{1,\dots,d\}$
such that $\phi_I$ touches the polycircle, there exists $D\geq 1$ such that, for all $\eta\in\TT^{|I|}$ and  all $\ovd\in(0,1)^{|I|},$
$$\left\{((1-\rho_k)e^{i\theta_k}):\ 0\leq \rho_k\leq \omega_{I,k}(\ovd),\ k\in P_I\textrm{ and }e^{i\theta}\in \phi_I^{-1}(S_{I}(\eta,\ovd))\right\}
\subset \phi_I^{-1}(S_I(\eta,D\ovd)).$$
\item For all $\xi\in\ovD,$ for all $I\subset\{1,\dots,d\}$ with $\phi_I(\xi)\in\TT^{|I|},$ there exist $D\geq 1$ and a neighbourhood $\mathcal U$ of $\xi$ in $\ovD$ such that, for all $\eta\in\TT^{|I|}$ and all $\ovd\in(0,1)^{|I|},$
$$\phi_I^{-1}(S_I(\eta,\ovd))\cap\mathcal U\subset \left\{((1-\rho_k)e^{i\theta_k}):\ 0\leq \rho_k\leq D\omega_{I,k}(\ovd),\ k\in P_I\textrm{ and }e^{i\theta}\in \phi_I^{-1}(S_{I}(\eta,D\ovd))\right\}.$$
\end{enumerate}
\end{lemma}

Moreover, even if we can't compute $\wik$, we can estimate its product as it is done in the following lemmas: 

\begin{lemma}[See Lemma $2.9$ of \cite{BD26}]\label{lem:produitwk}
Let $\phi\in\mathcal O(\DD^d,\DD^d)\cap \mathcal C^1(\ovD)$ and  $I\subset\{1,\dots,d\}$ be such that $\phi_I$ touches the polycircle. Then for all $\ovd\in(0,1)^{|I|}, $ 
\begin{equation}\label{eq:produitwk}
 \dsp_{k=1}^d\wik  \leq \dsp_{j \in I} \delta_j^{\frac{\m{P_I}}{\m{I}}}. 
\end{equation}

\end{lemma}

 \begin{lemma}[See Lemma $2.14$ of \cite{BD26}]\label{lem:wikcontinu}
Let $\phi\in\mathcal O(\DD^d,\DD^d)  \cap  \mathcal{C}^1(\ovD)$ and let $\beta_1,\beta_2\geq -1$. Assume that $C_\phi$ maps boundedly $A_{\beta_1}^2(\DD^d)$
into $A_{\beta_2}^2(\DD^d)$. Then there exists $D>0$ such that, for all $I\subset\{1,\dots,d\}$ such that $\phi_I$ touches the polycircle,
$$\prod_{k=1}^d \wik^{2+\beta_2}\leq D\prod_{j\in I}\delta_j^{2+\beta_1}.$$
\end{lemma}
\smallskip

\subsection{How to prove compactness}
Now, we are ready to state the result: 
\begin{theorem}\label{result_gen_comp}
Let $\beta_1, \ \beta_2\geq-1$.  Let $\phi \in \mathcal{O}(\DD^d,\DD^d)$ $ \cap  $ $\mathcal{C}^1(\ovD)$. Then, the following assertions are equivalent:  
\begin{enumerate}[(a)]
\item ${C_\phi : A_{\beta_1}^2(\DD^d) \rightarrow A_{\beta_2}^2(\DD^d)}$ is compact. \\
\item For all $I \subset \I{d}$ such that $\phi_I$ touches the polycircle and  for all $\veps >0, $ there exists $\gamma>0$ such that for all $\eta \in \TTI$ and all $\ovd \in (0,1)^{\m{I}},$ $\min\limits_{i\in I} \delta_i \leq \gamma $ implies 
$$\lambda_d(\lbrace \theta \in [-\pi,\pi)^d : \m{\phi_j(\theta)-\eta_j}<\delta_j,\  j\in I \rbrace) \dfrac{\prod\limits_{k=1}^d\omega_{I,k}(\ovd)^{1+\beta_2}}{\prod\limits_{j\in I}\delta_j^{2+\beta_1}}\leq \veps. $$
\item For all $\theta_0 \in \RR^d$, all $I \subset \I{d}$ such that  $\phi_I(\theta_0)\in \TTI$ and $I$ is maximal with respect to this property, there exists a neighbourhood $O$ of $\theta_0$ in $\RR^d$ such that for all $\veps >0, $ there exists $\gamma>0$ such that for all $\eta \in \TTI$ and all $\ovd \in (0,1)^{\m{I}},$ $\min\limits_{i\in I} \delta_i \leq \gamma $ implies 
\begin{equation}\label{eq:carlesonbord}
\lambda_d(\lbrace \theta \in O : \m{\phi_j(\theta)-\eta_j}<\delta_j,\ j\in I \rbrace)\dfrac{\prod\limits_{k=1}^d\omega_{I,k}(\ovd)^{1+\beta_2}}{\prod\limits_{j\in I}\delta_j^{2+\beta_1}}\leq \veps.
\end{equation}
\end{enumerate}
\end{theorem} 

\begin{proof}
We first observe that $b) \Rightarrow c)$ is immediate. We split the proof of the other implications in several cases, following $A_{\beta_1}^2$ (resp. $A_{\beta_2}^2$) is, or is not, the Hardy space. \\ 

\noindent \textit{Case 1: $C_\phi : A^2_{\beta_1}(\DD^d) \rightarrow A^2_{\beta_2}(\DD^d) $ with $\beta_1,\beta_2>-1$. } \\

\noindent $a) \Rightarrow  b)$ Let  $I \subset \I{d}$ be such $\phi_I$ touches the polycircle. Let also $\eta\in\TTI$ and $\ovd\in(0,1)^{|I|}.$ By Lemma \ref{lem:doubleinclusion}, there exists $D\geq 1$ such that 
$$E:=\left\{((1-\rho_k)e^{i\theta_k}):\ 0\leq \rho_k\leq \omega_{I,k}(\ovd),\ k\in P_I\textrm{ and }e^{i\theta}\in \phi_I^{-1}(S_{I}(\eta,\ovd))\right\}
\subset \phi_I^{-1}(S_I(\eta,D\ovd)).$$
Let us now extend $\ovd$ to $(0,2]^d$ by setting $\delta_j=2$ if $j\notin I$ and $\eta$ to $\TT^d$ by taking $\eta_j=1$ if $j\notin I.$
Then 
$$\phi_I^{-1}(S_I(\eta,D \ovd))= \phi^{-1}(S(\eta,D\ovd)).$$
We now use Lemma \ref{equiv} to get 
\begin{align*}
\lambda_d(\lbrace \theta \in [-\pi,\pi)^d : \m{\phi_j(\theta)-\eta_j}<\delta_j,\ j\in I \rbrace) \prod\limits_{k=1}^d\omega_{I,k}(\ovd)^{1+\beta_2} &\lesssim V_{\beta_2}(E)\\
&\lesssim V_{\beta_2}(\phi_I^{-1}(S_I(\eta,D \ovd)))\\
&\lesssim V_{\beta_2}(\phi^{-1}(S(\eta,D\ovd )))
\end{align*}
where the involved constants do not depend on $\eta$ and $\ovd.$ However, as the operator is compact, we have 
$$\forall \veps>0, \hspace{2mm} \exists \gamma >0, \hspace{2mm} \forall \ovd \in (0,2]^d,  \hspace{2mm} \forall \eta \in \TT^d, \hspace{2mm} \min\limits_{i = 1,\dots,d}\delta_i \leq\gamma \ \Rightarrow \ \dfrac{\vb{2}(\phi^{-1}(\sed))}{\prod\limits_{j=1}^d \delta_j^{2+\beta_1}} \leq \veps .$$
Thus, for all $\veps>0$,  there exists $\gamma >0$, such that for all $\ovd \in (0,2]^{\m{I}}$ and all $\eta \in \TTI$,  $\min\limits_{i \in I }\delta_i \leq\gamma $ implies \begin{align*}
\lambda_d(\lbrace \theta \in [-\pi,\pi)^d : \m{\phi_j(\theta)-\eta_j}<\delta_j,\  j\in I \rbrace) \prod\limits_{k=1}^d\omega_{I,k}(\ovd)^{1+\beta_2}  \lesssim  \veps\dsp_{j=1}^d\delta_j^{2+\beta_1} \lesssim \veps \dsp_{j\in I}\delta_j^{2+\beta_1}. 
\end{align*}
\smallskip

\noindent c) $\Rightarrow $ a) The proof will be divided into two steps. We first show that \eqref{eq:carlesonbord} implies a local Carleson type condition and then we conclude with a covering argument. \\
\noindent {\bf Fact:}
Let $\xi\in\ovD$ and $I\subset\{1,\dots,d\}$ be such that $\phi_I(\xi)\in\TTI$ and $I$ is maximal with respect to this property. Then there exists a neighbourhood $\mathcal U$ of $\xi$ in $\overline{\DD}^d$ such that, for all $\veps >0, $ there exists $\gamma = \gamma(\veps,\beta_1,\beta_2) >0$ such that for all $\eta\in \TTI$   and all $\ovd\in(0,1)^{|I|},\ \min\limits_{i \in I}\delta_i\leq\gamma $ implies  
$$V_{\beta_2}(\phi_I^{-1}(\Si)\cap \mathcal U) \leq \veps \dsp_{j\in I}\delta_j^{2+\beta_1}$$
where $\gamma$ does not depend on $\beta_1$ (resp. $\beta_2$) if $\beta_1\in(-1,0]$ (resp. $\beta_2\in (-1,0]$).

\noindent {\it Proof of the fact:}
Let us write $\xi=((1-\widetilde{\rho_k})e^{i\widetilde{\theta_k}})$ and observe that $\phi_I(e^{i\widetilde\theta})\in\TTI.$ Let $O$ be a neighbourhood of $\widetilde\theta$ be given by c) of Theorem \ref{result_gen_comp}. Let $D\geq 1$ and $\mathcal U_1$ be the neighbourhood of $\xi$ given by b) of Lemma \ref{lem:doubleinclusion}.
We set 
$$\mathcal U=\mathcal U_1\cap \{((1-\rho_k)e^{i\theta_k})_{k=1,\dots,d}: 0\leq\rho_k\leq 3/2,\ e^{i\theta}\in O\}.$$
Then $\mathcal U$ is a neighbourhood of $\xi$ and for all $\eta \in \TTI$ and all $\ovd \in \oii,$
$$\phi_I^{-1}(S_I(\eta,\ovd))\cap\mathcal U\subset \left\{((1-\rho_k)e^{i\theta_k}):\ 0\leq \rho_k\leq D\omega_{I,k}(\ovd),\ k\in P_I\textrm{ and }e^{i\theta}\in \phi_I^{-1}(S_{I}(\eta,D\ovd))\cap O\right\}.$$
We use Lemma \ref{equiv} and c) of Theorem \ref{result_gen_comp} to get : $\forall \veps >0 , \hspace{2mm} \exists \gamma >0, \hspace{2mm}\forall \eta \in \TTI, \hspace{2mm} \forall \ovd \in (0,1)^{\m{I}}, \hspace{2mm} {\min\limits_{i \in I}\delta_i \leq \gamma},$
\begin{align*}
V_{\beta_2}(\phi_I^{-1}(S_I(\eta,\ovd))\cap \mathcal U)&\leq
C_{\beta_2} \prod_{k=1}^d\omega_{I,k}(D\ovd)^{1+\beta_2}\lambda_d(\phi_I^{-1}(S_I(\eta,D\ovd))\cap O)\\
&\leq C_{\beta_2}D^{d(2+\beta_1)}C_{\beta_1}\veps \prod_{j \in I}\delta_j^{2+\beta_1}.
\end{align*}
The fact that  the constant $C_{\beta_1}$ and $C_{\beta_2}$ can be chosen independentely from $\beta_1$ and $\beta_2$  if $\beta_1$ and $\beta_2$ belong to $(-1,0] $ comes from Lemma $3.6$ of \cite{Baytridisc} and Lemma $8.1$ of \cite{BAYPOLY}. Thus $\gamma(\veps,\beta_1,\beta_2) = \gamma(\veps,\beta_2) $ (resp  $\gamma(\veps,\beta_1,\beta_2) = \gamma(\veps,\beta_1) $) when $\beta_1 \in (-1,0]$ (resp $\beta_2 \in (-1,0])$. 
\smallskip

We are now ready for the second half of $c)\implies a),$ the covering argument.
We introduce $$A_1 = \lbrace \xi \in \overline{\DD}^d:\  \phi(\xi) \in \TT^d\}. $$ 
Let $\xi\in A_1.$ By the fact, there exists $\mathcal U_{1,\xi}$ a neighbourhood of $\xi$ in $\overline{\DD}^d$ 
such that for all $\veps >0,$ there exists $\gamma_1 >0$ such that for all $\eta \in \TT^d$ and all $\ovd\in (0,1)^d$, $\min\limits_{i =1,\dots,d}\delta_i \leq\gamma_1$ implies 
$$V_{\beta_2}(\phi^{-1}(S(\eta,\ovd)) \cap \mathcal U_{1,\xi}) \leq 
\veps V_{\beta_1}(S(\eta,\ovd)).$$
Then, the sets $(\mathcal U_{\xi})_\xi$ form an open covering of $A_1$ which is compact so we can extract a finite covering $(\mathcal U_{1,k})_{k=1,\dots,s_1}$. For all $\eta \in \TT^d$ and $\ovd \in (0,1)^d $ with $\min\limits_{i =1,\dots,d}\delta_i\leq\gamma_1,$ we get 
$$\dfrac{\vb{2}\left(\phi^{-1}(S(\eta,\ovd)) \cap \bigcup\limits_{k=1}^{s_1}\mathcal U_{1,k}\right)}{\dsp_{j=1}^d\delta_j^{2+\beta_1}} \leq \sum\limits_{k=1}^{s_1}\dfrac{\vb{2}(\phi^{-1}(S(\eta,\ovd)) \cap \mathcal U_{1,k})}{\dsp_{j=1}^d\delta_j^{2+\beta_1}} \leq \left(\sum\limits_{k=1}^{s_1}1
\right)\veps. $$ 
Observe that $s_1$ does not depend on $\gamma_1$. Then, there exists $\varepsilon_1 \in (0,1)$ such that for all $\xi \in B_1 = \ovD \setminus \bigcup\limits_{k=1}^{s_1}\mathcal U_{1,k} $, we have $\m{\phi_j(\xi)}< 1-\varepsilon_1$ for some $j \in \lbrace 1,\dots,d\rbrace$. \\
We introduce 
$$A_2 = \lbrace \xi \in B_1:\  \exists I \subset \lbrace 1 ,\dots,d \rbrace,\  \m{I}  = d-1,\ \phi_I(\xi)\in \TT^{d-1} \rbrace.$$
Let $\xi \in A_2.$  There exists $I \subset \lbrace 1 ,\dots,d \rbrace$ with  $\m{I}  = d-1$ such that $ \phi_I(\xi)\in \TT^{d-1}$ and $\m{\phi_j(\xi)} < 1 - \varepsilon_1$ for $j \notin I$. So, applying again the fact, we get that there exists $\mathcal U_{2,\xi}$ a neighbourhood of $\xi$ in $\ovD$ 
such that for all $\veps >0,$ there exists $\gamma >0$ such that for all $\eta \in \TT^d$ and all $\ovd\in (0,1)^d$, $\min\limits_{i\in I}\delta_i \leq\gamma$ implies 
$$V_{\beta_2}(\phi_I^{-1}(S_I(\eta,\ovd)) \cap \mathcal U_{2,\xi}) \leq 
\veps \dsp_{j\in I }\delta_j^{2+\beta_1}.$$
Let $j \notin I$. If $\delta_j < \varepsilon_1$ then $\phi_j^{-1}(S(\eta,\ovd)) \cap \mathcal U_{2,\xi} = \emptyset$ for all $z$. Then, for all $\veps >0,$ there exists $\gamma_2 = \min(\veps_1,\gamma)  \in (0,\gamma_1]$ such that for all $\eta \in \TT^d$ and all $\ovd\in (0,1)^d$, $\min\limits_{i =1,\dots,d}\delta_i \leq\gamma_2$ implies 
\begin{align*}
\dfrac{V_{\beta_2}(\phi^{-1}(\sed) \cap \mathcal U_{2,\xi})}{\dsp_{j=1}^d \delta_j^{2+\beta_1}} \leq \dfrac{V_{\beta_2}(\phi_I^{-1}(\Si) \cap \mathcal U_{2,\xi})}{\dsp_{j \in I} \delta_j^{2+\beta_1}} \cdot \dfrac{1}{\veps_1^{2+\beta_1}} \leq D_{2,\xi}\veps.
\end{align*}
As $\dfrac{1}{\varepsilon_1^{2+\beta_1}}\leq \dfrac{1}{\varepsilon_1^2}$ when $\beta_1 \in (-1,0],$ we get that $D_{2,\xi}$ is independent of $\beta_i $ if $\beta_i \in (-1,0].$ Moreover, the sets $(\mathcal U_{2,\xi})_\xi$ form an open covering of $A_2$ which is compact so we can extract a finite covering $(\mathcal U_{2,k})_{k=1,\dots, s_2}$ and we have for all $\veps >0,$ there exists $\gamma_2  >0$ such that for all $\eta \in \TT^d$ and all $\ovd\in (0,1)^d$, $\min\limits_{i =1,\dots,d}\delta_i \leq\gamma_2$ implies 
$$\dfrac{V_{\beta_2}\left(\phi^{-1}(\sed) \cap \bigcup\limits_{k=1}^{s_2} \mathcal U_{2,k}\right)}{\dsp_{j=1}^d\delta_j^{2+\beta_1}} \leq \sum_{k=1}^{s_2}\dfrac{\vb{2}(\phi^{-1}(\sed\cap \mathcal U_{2,k}))}{\dsp_{j=1}^d\delta_j^{2+\beta_1}} \leq \left(\sum_{k=1}^{s_2}D_{2,k}\right)\veps.$$ 
Then, there exists $\varepsilon_2 \in (0 ,\varepsilon_1)$ such that for all $\xi \in B_2 = B_1 \setminus \bigcup\limits_{k=1}^{s_2}\mathcal U_{2,k}$, there exist two distinct $j_1, j_2 \in \lbrace 1,\dots,d\rbrace$ such that $\m{\phi_{j_1}(\xi)} < 1-\varepsilon_2$  and $\m{\phi_{j_2}(\xi)} < 1-\varepsilon_2$.
We repeat the reasoning with 
$$A_3 =  \lbrace \xi \in B_2:\ \text{ there exists } I \subset \I{d},\ \m{I} = d-2,\ \phi_I(\xi) \in \TT^{d-2} \rbrace$$ and then we define $B_3, A_4, \dots , B_{d-1}, A_d$. At the last step, there exists $\varepsilon_d \in (0,\varepsilon_{d-1})$ such that for all $\xi \in B_d = B_{d-1}\setminus \bigcup\limits_{k=1}^{s_d}\mathcal U_{d,k}$, we have $\m{\phi_j(z)} < 1-\varepsilon_d$ for all $j\in \lbrace 1,\dots,d\rbrace.$\\
Then, for all $\veps >0,$ there exists $\gamma_d \in (0,\gamma_{d-1}]$ such that for all $\eta \in \TT^d$ and all $\ovd\in (0,1)^d$, $\min\limits_{i =1,\dots,d}\delta_i \leq \gamma_d$ implies 
\begin{align*}
    \dfrac{\vb{2}(\phi^{-1}(\sed))}{\dsp_{j=1}^d\delta_j^{2+\beta_1}} \leq \sum_{l} \dfrac{\vb{2}\left(\phi^{-1}(\sed) \cap \bigcup\limits_{k}\mathcal U_{l,k} \right)}{\dsp_{j=1}^d\delta_j^{2+\beta_1}} &\leq \left(\sum_{k,l}D_{l,k}\right)\veps = D \veps,
\end{align*}
with $D>0$ and independent of $\beta_i $ if $\beta_i \in (-1,0].$ Notice  that  that $\gamma_d$ is also independent of $\beta_i$ for $\beta_i \in (-1,0]$. 
Thus, $C_\phi : \Abdi{1} \rightarrow \Abdi{2}$ is compact. \\

\noindent \textit{Case 2 : $C_\phi : A^2_{\beta_1}(\DD^d) \rightarrow H^2(\DD^d)$ with $\beta_1>-1.$} The proof is identical since the characterisation of compactness extends to $\beta_2=-1$. We replace everywhere $V_{\beta_2}$ by $\lambda_d.$ \\

\noindent \textit{Case 3 : $C_\phi :  H^2(\DD^d)  \rightarrow  A^2_{\beta_2}(\DD^d)$ with $\beta_2>-1.$}  $a) \Rightarrow b)$ The proof is again identical as the necessity of the condition holds on $H^2(\DD^d)$. To prove $c) \Rightarrow a)$ we are going to apply Theorem \ref{comp_Hardy_Bergman}.  By assumptions,  for all $\theta_0 \in \RR^d$, all $I \subset \I{d}$ such that  $\phi_I(\theta_0)\in \TTI$ and $I$ is maximal with respect to this property, there exists a neighbourhood $O$ of $\theta_0$ in $\RR^d$ such that for all $\veps >0, $ there exists $\gamma>0$ such that for all $\eta \in \TTI$ and all $\ovd \in (0,1)^{\m{I}},$ $\min\limits_{i\in I} \delta_i \leq \gamma $ implies $$
\lambda_d(\lbrace \theta \in O : \m{\phi_j(\theta)-\eta_j}<\delta_j,\ j\in I \rbrace)\dfrac{\prod\limits_{k=1}^d\omega_{I,k}(\ovd)^{1+\beta_2}}{\prod\limits_{j\in I}\delta_j}\leq \veps. $$ Let us consider $\beta_2' = \beta_2+d(\beta+1)$  for all  $\beta \in (-1,0]$. Then
 \begin{align*}
    \dfrac{\dsp_{k=1}^d\wik^{1+\beta_2'}}{\dsp_{j\in I}\delta_j^{2+\beta}} = \dfrac{\dsp_{k=1}^d\wik^{1+\beta_2+d(\beta+1)}}{\dsp_{j\in I}\delta_j^{2+\beta}} = \dfrac{\prod\limits_{k=1}^d\omega_{I,k}(\ovd)^{1+\beta_2}}{\prod\limits_{j\in I}\delta_j} \left(\dfrac{\prod\limits_{k=1}^d\omega_{I,k}(\ovd)^{d}}{\prod\limits_{j\in I}\delta_j}\right)^{1+\beta} 
    \leq  \dfrac{\prod\limits_{k=1}^d\omega_{I,k}(\ovd)^{1+\beta_2}}{\prod\limits_{j\in I}\delta_j},
 \end{align*}
because $\dfrac{\prod\limits_{k=1}^d\omega_{I,k}(\ovd)^{d}}{\prod\limits_{j\in I}\delta_j} \leq 1 $ by Lemma \ref{lem:produitwk}. Hence, for all $\theta_0 \in \RR^d$, all $I \subset \I{d}$ such that  $\phi_I(\theta_0)\in \TTI$ and $I$ is maximal with respect to this property, there exists a neighbourhood $O$ of $\theta_0$ in $\RR^d$ such that for all $\veps >0, $ there exists $\gamma>0$ such that for all $\eta \in \TTI$ and all $\ovd \in (0,1)^{\m{I}},$ $\min\limits_{i\in I} \delta_i \leq \gamma $ implies
$$ \lambda_d(\lbrace \theta \in O : \m{\phi_j(\theta)-\eta_j}<\delta_j,\ j\in I \rbrace)  \dfrac{\dsp_{k=1}^d\wik^{1+\beta_2'}}{\dsp_{j\in I}\delta_j^{2+\beta}} \leq \veps. $$ 
In the same way as in the proof of the first case of this theorem, one can show that 
$$ \forall \veps>0, \ \exists \gamma >0, \ \forall\eta\in \TT^d, \ \forall \ovd \in (0,1)^d, \  \forall \beta \in (-1,0],\ \min\limits_{i=1,\dots,d}\delta_i\leq \gamma \ \Rightarrow \ \vb{2'}(\phi^{-1}(\sed)) \leq \veps \dsp_{j=1}^d\delta_j^{2+\beta}. $$
Thus, by Theorem \ref{comp_Hardy_Bergman}, we deduce that $C_\phi : H^2(\DD^d) \rightarrow \Abd$ is compact. \\
\noindent \textit{Case 4 : $C_\phi :  H^2(\DD^d) \rightarrow H^2(\DD^d)$.} 
The proof of $a) \Rightarrow b) $ is once again identical. Now to prove $c) \Rightarrow a) $, we want to apply Theorem $8.2$ of \cite{Baytridisc}. By assumptions,  for all $\theta_0 \in \RR^d$, all $I \subset \I{d}$ such that  $\phi_I(\theta_0)\in \TTI$ and $I$ is maximal with respect to this property, there exists a neighbourhood $O$ of $\theta_0$ in $\RR^d,$
 there exists $\gamma>0$ such that for all $\eta \in \TTI$ and all $\ovd \in (0,1)^{\m{I}},$ $\min\limits_{i\in I} \delta_i \leq \gamma $ implies $$
\lambda_d(\lbrace \theta \in O : \m{\phi_j(\theta)-\eta_j}<\delta_j,\ j\in I \rbrace)\leq\prod\limits_{j\in I}\delta_j. $$ 
 Besides, for $\ovd \in (0,1)^d $ such that $\min\limits_{i\in I} \delta_i \geq \gamma ,$   we know that $\lambda_d(\lbrace \theta \in O : \m{\phi_j(\theta)-\eta_j}<\delta_j,\ j\in I \rbrace)$ is bounded because $O$ is a bounded neighbourhood of $\theta_0 \in \RR^d$ so there exists $C_1 >0$ such that $$\lambda_d(\lbrace \theta \in O : \m{\phi_j(\theta)-\eta_j}<\delta_j,\ j\in I \rbrace)\leq C_1 \dsp_{j\in I} \delta_j$$ (because $ \dsp_{j\in I} \delta_j \geq \gamma_1^{\m{I}} >0).$ Thus,  for $C= \max(1,C_1),$ we obtain that 
 for all $\theta_0 \in \RR^d$, all $I \subset \I{d}$ such that  $\phi_I(\theta_0)\in \TTI$ and $I$ is maximal with respect to this property, there exist a neighbourhood $O$ of $\theta_0$ in $\RR^d$ and $C>0$ such that for all $\eta \in \TTI$ and all $\ovd \in (0,1)^{\m{I}},$ $$
\lambda_d(\lbrace \theta \in O : \m{\phi_j(\theta)-\eta_j}<\delta_j,\ j\in I \rbrace)\leq C \prod\limits_{j\in I}\delta_j $$ and $C_\phi: H^2(\DD^d) \rightarrow H^2(\DD^d)$ is bounded. Now, we intend to prove the following fact:\\ 
\noindent {\bf Fact:} Let $\xi\in\ovD$ and $I\subset\{1,\dots,d\}$ be such that $\phi_I(\xi)\in\TTI$ and $I$ is maximal with respect to this property. Then there exists a neighbourhood $\mathcal U$ of $\xi$ in $\overline{\DD}^d$ such that, for all $\veps >0, $ there exists $\gamma >0$ such that for all $\eta\in \TTI,$ all $\beta \in (-1,0]$ and  all $\ovd\in(0,1)^{|I|},\ \min\limits_{i \in I}\delta_i\leq\gamma $ implies 
$$V_{\beta}(\phi_I^{-1}(\Si)\cap \mathcal U) \leq \veps \dsp_{j\in I}\delta_j^{2+\beta}.$$
As in the first case of the present theorem, we construct such a neighbourhood $\mathcal{U}$ and we have that for all $\veps >0, $ there exists $\gamma >0$ such that for all $\eta \in \TTI,  $ all $\ovd \in (0,1)^{\m{I}}, $ $\min\limits_{i \in I}\delta_i\leq  \gamma$ implies
\begin{align*}
V_\beta(\phi^{-1}_I(\Si)\cap \mathcal U) \leq C_\beta \dsp_{k=1}^d\wik^{1+\beta}\lambda_d(\phi^{-1}_I(S_I(\eta, D\ovd))) &\leq C_\beta\veps \dsp_{j \in I}\delta_j\dsp_{k=1}^d\wik^{1+\beta} \\
 &\leq C_\beta \veps \dsp_{j\in I} \delta_j^{2+\beta} \left(\dfrac{\dsp_{k=1}^d\wik}{ \dsp_{j\in I} \delta_j}\right)^{1+\beta}.  
\end{align*}
However, as $C_\phi : H^2(\DD^d) \rightarrow H^2(\DD^d) $ is bounded, by Lemma \ref{lem:wikcontinu}, we know that there exists $D_1>0$ such that $\dsp_{k=1}^d\wik \leq D_1\dsp_{j\in I}\delta_j. $ Hence, 
$$\left(\frac{\dsp_{k=1}^d\wik}{ \dsp_{j\in I} \delta_j}\right)^{1+\beta} \leq D_1^{1+\beta} \leq  \left\{ \begin{tabular}{ccc}
     $1$& \text{ if }& $D_1 \in (0,1)$ \\
     $D_1^2$ &\text{ if } &$ D_1 \geq 1$
\end{tabular}.
\right.
$$

Thus, $$V_\beta(\phi^{-1}_I(\Si)\cap \mathcal U) \leq C_\beta \veps C_{D_1} \dsp_{j \in I}\delta_j^{2+\beta} = C\veps \dsp_{j \in I}\delta_j^{2+\beta},$$
with $C>0$  independent from $\beta \in (-1,0] $ (see Lemma $3.6$ of \cite{Baytridisc} and Lemma $8.1$ of \cite{BAYPOLY}) and $\gamma$ is also independent from $\beta \in (-1,0]$.  Thus, there exists $\mathcal{U} $ a neighbourhood of $\xi \in \DD^d $, such that for all $\veps >0, $ there exists $\gamma >0$ such that for all $\eta \in \TTI,  $ all $\beta \in (-1,0]$ and  all $\ovd \in (0,1)^{\m{I}}, $ $\min\limits_{i \in I}\delta_i\leq \gamma$ implies $V_\beta(\phi^{-1}_I(\Si)\cap \mathcal U)  \leq \veps \dsp_{j \in I}\delta_j^{2+\beta}$ . Then, by the same covering argument as in the first case of the proof of the current theorem, one can show that 
$$ \forall \veps >0, \ \exists \gamma > 0, \ \forall \eta \in \TT^d, \ \forall \ovd\in (0,1)^d,\ \forall \beta \in (-1,0],\ \min\limits_{i= 1,\dots,d} \delta_i \leq\gamma  \ \Rightarrow V_\beta(\phi^{-1}(\sed)) \leq \veps \dsp_{j=1}^d \delta_j^{2+\beta}. 
$$
Applying Theorem $8.2$ of \cite{Baytridisc}, we deduce that $C_\phi : H^2(\DD^d) \rightarrow H^2(\DD^d)  $ is compact. 
\end{proof} 

We conclude this section with two corollaries that provide conditions for compactness and non-compactness using only the product of the $\delta_i$, $i=1,\dots,d$, rather than $\wik$.

\begin{corollary}\label{coro_compacite}
Let $\beta_1, \ \beta_2\geq -1 $ and $\phi \in \mathcal O(\DD^d,\DD^d) \cap \mathcal C^1(\overline \DD^d). $
\begin{enumerate}[(a)]
    \item Let $\theta_0 \in \RR^d$ and $I \subset \lbrace 1, \dots, d\rbrace $ be such that $\phi_I(e^{i\theta_0}) \in \TTI$ and 
 there exist $C>0$ and $O$ a neighbourhood of $\theta_0$ in $\RR^d$ such that for all $\eta \in \TT^{\m{I}}$, $\ovd \in (0,1)^{\m{I}}$ we have 
 \begin{equation}\label{eq:cont1}
\lambda_d(\lbrace \theta \in O : \m{\phi_j(\theta)- \eta_j } < \delta_j,\ j \in I\rbrace) \leq C \left(\displaystyle\prod_{j\in I}\delta_j\right)^\alpha, 
\end{equation}
with $\alpha + \dfrac{\m{P_I}}{\m{I}}(1+\beta_2) > 2+\beta_1$. \\
Then for all $\veps >0,$ there exists $\gamma >0$ such that for all $\eta \in \TTI$ and all $\ovd \in (0,1)^{\m{I}}, $ $\min\limits_{i\in I}\delta_i \leq\gamma $ implies  
$$\lambda_d(\lbrace \theta \in O : \m{\phi_j(\theta)- \eta_j } < \delta_j,\ j \in I\rbrace) \leq \veps  \dfrac{\prod\limits_{j\in I}\delta_j^{2+\beta_1}}{\prod\limits_{k=1}^d\omega_{I,k}(\ovd)^{1+\beta_2}} .$$
\item Assume that for all $\theta_0 \in \RR^d$ and all $I \subset \lbrace 1, \dots, d\rbrace $ such that $\phi_I(e^{i\theta_0}) \in \TTI$ and $I$ is maximal with respect to this property, there exist $C>0$ and  $O$ a neighbourhood of $\theta_0$ in $\RR^d$ such that for all $\eta \in \TT^{\m{I}}$ and all $\ovd \in (0,1)^{\m{I}},$ we have 
 \begin{equation}\label{eq:cont}
\lambda_d(\lbrace \theta \in O : \m{\phi_j(\theta)- \eta_j } < \delta_j,\ j \in I\rbrace) \leq C \left(\displaystyle\prod_{j\in I}\delta_j\right)^\alpha, 
\end{equation}
with $\alpha + \dfrac{\m{P_I}}{\m{I}}(1+\beta_2) > 2+\beta_1$. Then $C_\phi : \Abdi{1}\rightarrow \Abdi{2} $ is compact.  
\end{enumerate} 
\end{corollary}

\begin{proof}
As $\alpha + \dfrac{\m{P_I}}{\m{I}}(1+\beta_2) > 2+\beta_1$, there exists $\Delta >0, $ such that $\alpha =   2+\beta_1 - \dfrac{\m{P_I}}{\m{I}}(1+\beta_2) + \Delta$. Then by Lemma \ref{lem:produitwk}, we have 
$$\dsp_{j\in I}\delta_j^\alpha =\dsp_{j\in I}\delta_j^{2+\beta_1 - \frac{\m{P_I}}{\m{I}}(1+\beta_2) + \Delta}  \leq \dsp_{j\in I} \delta_j^{\Delta} \cdot \dfrac{\dsp_{j\in I}\delta_j^{2+\beta_1}}{\dsp_{k=1}^d \wik^{1+\beta_2}}. $$
Thus, 
$$\lambda_d(\lbrace \theta \in O : \m{\phi_j(\theta)- \eta_j } < \delta_j,\ j \in I\rbrace)\leq  C  \dsp_{j\in I} \delta_j^{\Delta} \cdot \dfrac{\dsp_{j\in I}\delta_j^{2+\beta_1}}{\dsp_{k=1}^d \wik^{1+\beta_2}}$$ 
and for all $\veps >0,$ there exists $\gamma >0$ such that for all $\eta \in \TTI$ and  all $\ovd \in (0,1)^{\m{I}},\ \min\limits_{i\in I}\delta_i \leq\gamma $ implies  
$$\lambda_d(\lbrace \theta \in O : \m{\phi_j(\theta)- \eta_j } < \delta_j,\ j \in I\rbrace) \leq \veps  \dfrac{\prod\limits_{j\in I}\delta_j^{2+\beta_1}}{\prod\limits_{k=1}^d\omega_{I,k}(\ovd)^{1+\beta_2}} .$$
The second point now follows from the first one and an application of Theorem \ref{result_gen_comp}. 
\end{proof}

\begin{corollary}\label{coro_noncom}
Let $\beta_1,\ \beta_2\geq -1.$ Let $\phi \in\mathcal O( \DD^d,\DD^d)\cap \mathcal{C}^1(\ovD)$. Assume that there exist $I \subset \lbrace 1,\dots,d\rbrace$ such that $\phi_I$ touches the polycircle and 
$\eta \in \TTI$ such that for all $\delta \in (0,1)$ 
$$\lambda_d(\lbrace \theta \in [-\pi,\pi)^d : \m{\phi_j(e^{i\theta})-\eta_j}<\delta,\  j \in I\rbrace) \geq  f(\delta)$$
where 
$$\limsup_{\delta\to 0}\frac{f(\delta)}{\delta^{|I|(2+\beta_1)-|P_I|(1+\beta_2)}}>0.$$
Then $C_\phi : A^2_{\beta_1}(\DD^d) \to A^2_{\beta_2}(\DD^d)$ is not compact. 
This holds in particular with $f(\delta)=\delta^{\alpha |I|}$ when 
$\alpha + \dfrac{\m{P_I}}{\m{I}}(1+\beta_2) \leq 2 +\beta_1.$
\end{corollary}

\begin{proof}
When $\delta_1=\cdots=\delta_d=\delta,$ inequality \eqref{eq:carlesonbord} reads 
$$\lambda_d(\lbrace \theta \in [-\pi,\pi)^d : \m{\phi_j(\theta)-\eta_j}<\delta,\   j \in I\rbrace) \frac{\delta^{|P_I|(1+\beta_2)}}{\delta^{|I|(2+\beta_1)}}\leq \veps.$$ However, $\limsup\limits_{\delta\to 0}\frac{f(\delta)}{\delta^{|I|(2+\beta_1)-|P_I|(1+\beta_2)}}>0$ means that there exists $\veps>0 $ such that  for all $\gamma >0, $ there exists $\delta \leq \gamma $ such that $\frac{f(\delta)}{\delta^{|I|(2+\beta_1)-|P_I|(1+\beta_2)}}\geq \veps.  $ Thus our assumptions implies that there exist $\veps >0$ such that for all $\gamma >0$, there exists $\delta\leq \gamma $ such that 
$$\lambda_d(\lbrace \theta \in [-\pi,\pi)^d : \m{\phi_j(\theta)-\eta_j}<\delta,\   j \in I\rbrace) \frac{\delta^{|P_I|(1+\beta_2)}}{\delta^{|I|(2+\beta_1)}}\geq \veps.$$
 We deduce that $C_\phi : \Abdi{1} \rightarrow \Abdi{2}$ is not compact by contrapositive of the implication $a) \Rightarrow b)$ of Theorem \ref{result_gen_comp}. 

\end{proof}
 \subsection{Applications}
Theorem \ref{result_gen_comp} has several applications. The first one is a stability  result for compactness:  

\begin{theorem}\label{thm:stability}
Let $\phi \in \mathcal{O}(\DD^d,\DD^d) \cap  \mathcal{C}^1(\ovD)$ and $\beta_1, \beta_2\geq -1$. Assume that $C_\phi$ compactly maps  $A_{\beta_1}^2(\DD^d)$ into $A_{\beta_2}^2(\DD^d)$. Then for all $\beta'_1\geq \beta_1,$ $C_\phi$ compactly maps  $A_{\beta_1'}^2(\DD^d)$ into $A_{\beta_2'}^2(\DD^d)$ where 
$$\beta_2'=\frac{\beta_1'(\beta_2+2)+2(\beta_2-\beta_1)}{\beta_1+2}.$$
In particular, if $C_\phi$ is compact on $A_{\beta_1}^2(\DD^d),$ then it is compact on $A_{\beta_1'}^2(\DD^d).$
\end{theorem}

\begin{proof}
The proof is the same as the proof of Theorem $2.13$ of \cite{BD26}. 
\end{proof}

The next result gives a characterisation of compactness for some classes of functions:

\begin{theorem}\label{thm:examples}
Let $\beta_1,\beta_2\geq -1.$ Let $\varphi\in\mathcal O(\DD,\DD)\cap \mathcal C^1(\overline\DD)$. Assume that, for any $z\in \overline\DD,$ $|\varphi(z)|=1$ if and only if $z=1$ and that there exist $\kappa>1$ and $c>0$ such that
$$|\varphi(e^{i\theta})|=1-c|\theta|^\kappa +o(|\theta|^\kappa).$$ Let $d\geq k\geq 1$ and let us set
$$\phi(z)=(\varphi(z_1)\cdots \varphi(z_k)z_{k+1}\cdots z_{d},\dots,\varphi(z_1)\cdots \varphi(z_k)z_{k+1}\cdots z_{d},0,\dots,0)$$ where the first $q$ coordinates, $q\in\{1,\dots,d\}$, are identical. Then $C_\phi : \Abd \rightarrow \Abd $ is compact if and only if 
$$d\beta_2 > \left(2q-d-1-\frac k\kappa\right)+q\beta_1.$$
\end{theorem}

\begin{proof}
The proof is the same as Theorem $2.17$ of \cite{BD26}
\end{proof}

\vspace{5mm}

\section{A geometric characterisation of compactness}

This part is dedicated to the proofs of Theorem $1.1$ and $1.2$.
\subsection{Technical lemmas}
First, we need some technical lemmas linking $\dsp_{k = 1}^d \wik $ and $\dsp_{j\in I}\delta_j  $ and a new definition.
\begin{definition}
Let $d\geq 1$ and $ j\in \I{d}, $ we define $T_{j,d} = \lbrace z \in \ovD : \exists I\subset\I{d},\ \m{I} = j,\  \m{z_k} = 1, \ \forall k \in I  \rbrace.  $  Notice that $\TT^d = T_{d,d} ,$ $T_{d+1,d} = \emptyset$ and  for all $j = 1,\dots,d-1,\ T_{j+1,d}\subset T_{j,d}.$ 	
\end{definition}

\begin{lemma}\label{max_principle}
Let $\phi \in \mathcal{O}(\DD^d,\DD^d) \cap \mathcal{C}^1(\ovD)$ be such that there exist $I \subset \I{d}$ and  $\xi \in \ovD$ such that $\phi_I(\xi ) \in \TTI$. Assume $\m{P_I} = k. $ Then $\xi \in  T_{k,d}$.
\end{lemma}

\begin{proof}
Let $I \subset \I{d}$ and  $\xi = (\xi_1,\dots,\xi_d) \in \ovD$ be such that $\phi_I(\xi ) \in \TTI$. For $j \in I$, we consider $\phi_j : \DD^d \rightarrow \DD$. We know that $\phi_j(\xi) \in \TT$. Then for $p \in P_j$, we consider $\phi_{p,j} : z_p \in \DD \mapsto \phi_j(\xi_1,\dots,\xi_{p-1},z_p,\xi_{p+1},\dots,\xi_d)$. $\phi_{p,j} $ is holomorphic on $\DD$ by Weierstrass' theorem and $\phi_{p,j}(\xi_p) \in \TT$. Thus by the maximum principle, we deduce that $\xi_p\in \TT$. We can do the same for every $p \in P_j$ and for every $j \in  P_I$ thus $\xi \in T_{k,d}$. 
\end{proof}

\begin{lemma}\label{lemme_recurrence}
Let $d \geq 2$. Let $\phi \in \mathcal{O}(\DD^d,\DD^d) \cap \mathcal{C}^1(\ovD).  $ Assume that
$$\forall  j= 1,\dots,d-1, \ \phi(\ovD \setminus T_{j+1,d}) \cap T_{j,d} = \emptyset. $$  
Let $I\subset \I{d}$ be  such that $\phi_I$ touches the polycircle with $\m{I}\leq d-1$. Then  for all $r \in \I{\m{I}}$ and  all $ \ovd \in (0,1)^{\m{I}}  $ such that $\delta_1 \leq \dots \leq  \delta_{\m{I}}, $
\begin{equation}\label{eq_lemme}
\dsp_{k \in P_1 \cup \dots\cup P_r}\wik  \leq  \left\{ 
\begin{array}{rcl}
  \delta_1^2\delta_2\cdots \delta_r^n &\textrm{ if } r > 1 \\
\delta_1^{\m{P_1}}& \textrm{ if  } r=1 
\end{array}
\right.
\end{equation}

where $n = \m{P_1 \cup \dots \cup P_r} - r. $
\end{lemma}

\begin{proof}
Let $I \subset \I{d}$ be  such that $\phi_I$ touches the polycircle and $\m{I} \leq d-1$, our assumptions imply that  \begin{equation}\label{hypothese}
\forall J \subset I ,\  \m{P_J} > \m{J}
\end{equation}
so $n$ is always greater than $1$. Let us explain this inequality. To do that, we are going to argue by contrapositive. Let $J \subset I $ with $\m{J}  =k \leq d-1$ and $\m{P_J} = k-p $ with $p\in \{0,\dots,k-1\}$. We have that $\phi_J$ touches the polycircle. Then there exists $\xi \in \ovD$ such that $\phi_J(\xi) \in \TT^k$. Thus, by Lemma \ref{max_principle}, $\xi \in T_{k-p,d}$. Taking $\widetilde{\xi} = \left\{ \begin{array}{rlc}
\xi_k \in \TT & \text{ if } k \in P_J\\
0 &\text{ othertwise}

\end{array}
\right.$, we have that there exists $\xi \in T_{k-p,d} \setminus T_{k-p+1,d}$ such that $\phi(\xi) \in  T_{k,d} \subset T_{k-p,d}$ and our condition with $j=k-p$ is not satisfied. Thus, $\m{P_J} > \m{J}$.\\

Let  $\ovd \in (0,1)^{\m{I}}$ be such that $\delta_1\leq \dots \leq \delta_{\m{I}}. $ We will proceed by induction on $r \in \I{\m{I}}. $ For  $r  = 1, $ the result is obvious. For $r= 2$, we have 
$$ \dsp_{k \in P_1 \cup P_2} \wik \leq \dsp_{k \in P_1} \wik \dsp_{k \in P_2\setminus P_1} \wik \leq \delta_1^{\m{P_1}} \delta_2^{\m{P_2\setminus P_1}} \leq \delta_1^2 \delta_2^{\m{P_2\setminus P_1}+\m{P_1}-2} = \delta_1^2 \delta_2^{\text{card}(P_1\cup P_2)-2}. $$ Now, let $r\in \lbrace 2,\dots ,\m{I}-1\rbrace$ and assume that $(\ref{eq_lemme})$ is satisfied for $r$ and let us prove it for $r+1$.
We have $$P_1\cup \dots \cup P_{r+1} = P_1\cup \dots \cup P_r \sqcup (P_{r+1}\setminus (P_1\cup \dots \cup P_r) )$$ and $\text{card}(P_1\cup \dots \cup P_{r+1}) = \text{card}(P_1\cup \dots \cup P_r) + \text{card}(P_{r+1}\setminus (P_1\cup \dots \cup P_r)) . $ Therefore, 
\begin{align*}
\dsp_{k \in P_1\cup \dots \cup P_{r+1}} \wik = \dsp_{k \in P_1\cup \dots \cup P_{r}} \wik \dsp_{k \in  P_{r+1}\setminus ( P_1\cup \dots \cup P_{r})} \wik &\leq \delta_1^2\delta_2\cdots\delta_r^n \delta_{r+1}^{\text{card}(P_{r+1}\setminus (P_1\cup \dots \cup P_r)) } \\
&\leq \delta_1^2\delta_2\cdots\delta_r\delta_{r+1}^{n-1+\text{card}(P_{r+1}\setminus (P_1\cup \dots \cup P_r)) },  
\end{align*}
with 
\begin{align*}
    n-1+\text{card}(P_{r+1}\setminus (P_1\cup \dots \cup P_r)) & =  \text{card}(P_1\cup \dots \cup P_r) -r-1 + \text{card}(P_{r+1}\setminus (P_1\cup \dots \cup P_r))  \\
    &=\text{card}(P_1\cup \dots \cup P_{r+1})-(r+1).  
\end{align*}
\end{proof}

\begin{corollary}\label{lemme_comp}
Let $d\geq 2$. Let $\phi \in \mathcal{O}(\DD^d,\DD^d) \cap \mathcal{C}^1(\ovD).$ Assume that  $$\forall j = 1,\dots, d,\  \phi(\ovD \setminus T_{j+1,d}) \cap T_{j,d} = \emptyset. $$
Let $I \subset \I{d} $ be such that $\phi_I$ touches the polycircle then for all $\ovd \in \oii, $
\begin{equation}\label{eq_coro_comp}
\dfrac{\dsp_{k=1}^d\wik}{\dsp_{j\in I }\delta_j} \leq \min\limits_{ j \in I} \delta_j. 
\end{equation}
\end{corollary}

\begin{proof}
Let $I \subset\I{d}$ be such that $\phi_I$ touches the polycircle and let $\ovd\in (0,1)^{\m{I}}.$ Without loss of generality, we may assume $\delta_1 \leq \dots \leq \delta_{\m{I}}. $ Then applying Lemma \ref{lemme_recurrence}, we get for $\m{I} \geq 2, $
$$\dfrac{\dsp_{k=1}^d \wik}{ \dsp_{j\in I }\delta_j} \leq \dfrac{\delta_1^2\delta_2\cdots\delta_{\m{I}}^n}{\delta_1\delta_2\cdots \delta_{\m{I}}} \leq \delta_1\delta_{\m{I}}^{n-1} \leq \delta_1 = \min\limits_{j\in I} \delta_j.$$
If $\m{I} = 1 , $ for example $I = \{1\}, $ we have 
$$\dfrac{\dsp_{k=1}^d \wik}{ \dsp_{j\in I }\delta_j}   = \dfrac{\delta_1^{\m{P_1}}}{\delta_1} = \delta_1^{\m{P_1}-1} \leq \delta_1,$$
because $(\ref{hypothese})$ implies $\m{P_1}-1 \geq 1. $
\end{proof}

  \begin{lemma}\label{lemme_cont}  
Let $d\geq 2.$  Let $\phi \in \mathcal{O}(\DD^d,\DD^d) \cap \mathcal{C}^1(\ovD) $ . Let $I\subset \I{d} $ be such that $\phi_I$ touches the polycircle and $\m{I} \leq d-1. $  Assume that $1 \leq \m{I } < \m{P_I} \leq d$.   Then for all $\ovd \in \oii,$
$$ \dfrac{\dsp_{j\in I } \delta_j^{d-1}}{\dsp_{k=1}^d \wik^{d-2}} \geq \min\limits_{j \in I} \delta_j .  $$
\end{lemma}

\begin{proof}
Let $\delta \in \oii. $ By Lemma \ref{lem:produitwk}, we know that $ \dsp_{k=1}^d\wik  \leq \dsp_{j \in I} \delta_j^{\frac{\m{P_I}}{\m{I}}}.$ 
Thus 
 \begin{align*}
\dfrac{\dsp_{j\in I}\delta_j^{d-1}}{\dsp_{k=1}^d\wik^{d-2}} \geq \dfrac{\dsp_{j\in I}\delta_j^{d-1}}{\dsp_{j\in I}\delta_j^{\frac{(d-2)\m{P_I}}{\m{I}} }} = \dsp_{j\in I} \delta_j^{d\left(1-\frac{\m{P_I}}{\m{I}}\right) -1 + 2\frac{\m{P_I}}{\m{I}}} \geq \min_{j\in I} \delta_j^{d(\m{I}-\m{P_I}) - \m{I}+2\m{P_I}}. 
 \end{align*}
Moreover, 
\begin{align*}
 \min_{j\in I} \delta_j^{d(\m{I}-\m{P_I}) - \m{I}+2\m{P_I}} \geq \min_{j\in I}\delta_j & \Leftrightarrow 1\geq d(\m{I} - \m{P_I})-\m{I} -2 \m{P_I} \\ 
 &\Leftrightarrow d(\m{P_I} - \m{I}) \geq \m{P_I} - \m{I} +\m{P_I}-1 \\
 &\Leftrightarrow d \geq 1+ \dfrac{\m{P_I}-1}{\m{P_I}-\m{I}}.
\end{align*}
The last inequality being always true because $1 \leq \m{I} < \m{P_I} \leq d.$ 
 \end{proof}

\subsection{The results}

And now our boundedness criterion: 
\begin{theorem}\label{thm_cont_geom}
Let $d \geq 2  $  and $\beta \geq d-3$. Let $\phi \in \mathcal{O}( \DD^d,\DD^d) \cap \mathcal{C}^1(\ovD). $ Then, $C_\phi : A^2_{\beta}(\DD^d) \rightarrow A^2_{\beta}(\DD^d)$ is bounded if and only if 
\begin{enumerate}[(a)]
    \item For all $ j = 1,\dots,d-1,\ \phi(\ovD \setminus T_{j+1,d} ) \cap T_{j+1,d} = \emptyset$
    \item For all  $I \subset \I{d}, \ \m{I } = j$ and all $ \xi \in T_{j,d} \setminus T_{j+1,d}$ such that $ \phi_I(\xi) \in \TT^j $ then $ \left( \dfrac{\partial\phi_i(\xi)}{\partial z_k }\right)_{i \in I, \ k \in P_I} $is invertible. 
\end{enumerate}

\end{theorem}

\begin{proof}
We  will prove the result on $\Abd$ for $\beta = d-3$ and applying Theorem $2.13$ of \cite{BD26}, we can deduce that the operator is bounded  on any $\Abd$ for $\beta >d-3$.\\
Let us first observe that in condition $(b)$, we have $\m{P_I}  =\m{I}$. Indeed, let $I\subset \I{d}$, $\m{I} = j$ and $\xi \in T_{j,d}\setminus T_{j+1,d}$ be such that $\phi_I(\xi) \in \TT^j$. If $\m{P_I}> \m{I}$, for example $\m{P_I} = j+1$, then as $\phi_I(\xi) \in \TT^j,$ by Lemma \ref{max_principle}, we have $\xi \in T_{j+1,d}$ which is a contradiction because $\xi \in T_{j,d}\setminus T_{j+1,d}$. Now, if $\m{P_I} < \m{I}$, for example $\m{P_I} = j-1$, then by Lemma \ref{max_principle}, we have $\xi  \in T_{j-1,d} \subset \ovD\setminus T_{j,d}.$ Thus, we have $\phi(\ovD \setminus T_{j,d}) \cap T_{j,d} \neq \emptyset$ and condition $(a)$ is not satisfied. In conclusion, $\m{P_I}  =\m{I}$.\\

First let us show that the conditions are sufficient. We intend to apply Theorem $2.10$ of \cite{BD26}.  Let $\xi \in \ovD $ and let $I \subset \I{d}$ of maximal cardinality such that $\phi_I(\xi) \in \TTI$. Firstly, if $\m{I} = d. $ Then, by assumption a) for $j= d-1$, we have $\xi = e^{i\theta_0} \in \TT^d$ and by b) the derivative $d\phi(\xi)$ is invertible thus applying Lemma $5.3$ of \cite{BD26} we get that there exist $C>0$ and $O $ a neighbourhood of $\theta_0$ in $\RR^d$ such that for all $\eta \in \TT^d$ and all $\ovd \in (0,1)^d, $ 
$$ \lambda_d(\lbrace \theta \in O : \m{ \phi_j(\theta) - \eta_j }< \delta_j, \ j = 1,\dots,d \rbrace ) \leq C \dfrac{\dsp_{j=1}^d\delta_j^{d-1}}{\dsp_{k=1}^d \wik^{d-2}}.  $$
Secondly, if $\m{I } = p \in \I{d-1}. $ Then there exists $I \subset \I{d}  $ such  that $\m{I} = p$ and $\phi_I(\xi) \in  \TT^p$. But then by a) $\xi \in T_{p,d}$ and $\m{P_I} \geq p. $ \\
If $\m{P_I} = p$ meaning $\xi  = re^{i\tu} \in T_{p,d}\setminus T_{p+1,d}. $ By    b), $\left( \dfrac{\partial\phi_i(\xi)}{\partial z_k }\right)_{i \in I, \ k \in P_I}$ is invertible thus applying Lemma $5.3$ of \cite{BD26} we get that there exist $C>0$ and $O$ a neighbourhood of $\tu $ in $\RR^d$ such that for all $\eta \in \TT^p $ and all $ \ovd \in \oii , $
$$\lambda_d(\{ \theta \in O : \m{\phi_j(\theta) - \eta_j } <\delta_j,\ j \in  I \} )\leq C \dfrac{\dsp_{j \in I } \delta_j^{d-1}}{\dsp_{k = 1 }^d\wik^{d-2}}.$$
If $\m{P_I} >p $ then applying Lemma \ref{lemme_cont}, we get $\dfrac{\dsp_{j \in I } \delta_j^{d-1}}{\dsp_{k = 1 }^d\wik^{d-2}} \geq \min\limits_{j \in I } \delta_j . $ Let $\xi = r_2e^{i\theta_2} \in T_{p+1,d}$. Applying Lemma 2.17 of \cite{BD26}, we get that there exists $O_1$ a neighbourhood of $\theta_2 \in \RR^d$ such that 
$$\lambda_d(\{ \theta \in O_1 : \m{\phi_j(\theta) - \eta_j } <\delta_j,\ j \in  I \} ) \lesssim \min\limits_{j \in I }\delta_j . $$ We deduce that there exists $C>0$ such that for all $\eta \in \TTI$ and all $\ovd \in \oii, $
 $$\lambda_d(\{ \theta \in O_1 : \m{\phi_j(\theta) - \eta_j } <\delta_j,\ j \in  I \})  \leq C\dfrac{\dsp_{j \in I } \delta_j^{d-1}}{\dsp_{k = 1 }^d\wik^{d-2}} .  $$ 
Finally, we showed that for all $\theta_0 \in \RR^d$ and all $I \subset\I{d} $ such that $\phi_I(\theta_0) \in \TTI$, there exist $C>0$ and $O$ a neighbourhood of $\theta_0 $ in $\RR^d $ such that for all $\eta \in \TTI$ and all $\ovd \in \oii$, 
 $$\lambda_d(\{ \theta \in O  : \m{\phi_j(\theta) - \eta_j } <\delta_j,\ j \in  I \})  \leq C\dfrac{\dsp_{j \in I } \delta_j^{d-1}}{\dsp_{k = 1 }^d\wik^{d-2}}  . $$
This implies that  $C_\phi : A^2_{d-3}(\DD^d) \rightarrow A^2_{d-3}(\DD^d)$ is bounded. \\

Now let us show that the conditions are necessary. We are going to argue by contrapositive. Let us assume either
\begin{enumerate}
    \item $ \exists j \in \I{d-1}, \ \phi(\ovD\setminus T_{j+1,d} ) \cap T_{j+1,d} \neq \emptyset  $ or
    \item $\exists j \in \I{d}, \ I \subset \I{d}, \ \m{I } = j, \ \exists \xi \in T_{j,d} \setminus T_{j+1,d} $ such that $\phi_I(\xi) \in \TTI$ and $\left( \dfrac{\partial\phi_i(\xi)}{\partial z_k }\right)_{i \in I, \ k \in P_I} $is not invertible. 
\end{enumerate}
First assume $(1). $ Let $\xi \in \ovD \setminus T_{j+1,d}$ and $I \subset \I{d}$ with $\m{I} = j+1$ be such that $\phi_I(\xi ) \in \TT^{j+1}.$ But then $\m{P_I} \leq j$ because if $\m{P_I} \geq j+1 $, by  Lemma \ref{max_principle}, $\xi \in T_{j+1,d}$ which is excluded. Let $\delta \in (0,1) . $ Then $[-\delta,\delta]^{\m{P_I}}\times [-\pi,\pi]^{d-\m{P_I}} \subset \{ \theta \in [-\pi, \pi)^d : \m{\phi_j(\theta) - 1 } <\delta, \  j\in I\}$ and 
$$ \lambda_d([-\delta,\delta]^{\m{P_I}}\times [-\pi,\pi]^{d-\m{P_I}} ) \gtrsim \delta^{\m{P_I}} = \delta^{\m{P_I}/\m{I}* \m{I}}. $$ Applying Corollary $2.13$ of \cite{BD26}, we get that $C_\phi$ is not bounded on any $\Abd.$\\
Now assume $(2)$.  Let $ j \in\I{d}, $ $I \subset \I{d}$ with $\m{I} = j$  and $\xi \in T_{j,d}\setminus T_{j+1,d}$ be such that $\phi_I(\xi) \in \TT^j$ and $ \left( \dfrac{\partial\phi_i(\xi)}{\partial z_k }\right)_{i \in I, \ k \in P_I} $is not invertible. Then $\m{P_I} = j$ because if $\m{P_I} \geq j+1$, by Lemma \ref{max_principle}, we have $\xi \in T_{j+1,d}$ which is excluded. Without loss of generality, we may assume that $I = \I{j}$ and $\xi = (1,\dots,1) = \phi_I(\xi).$  For $k \in I,$ we can write $\phi_k(\theta) = 1 +  L_k(\theta) + O(\tu^2+\dots + \theta_j^2).$  As the matrix $\left( \dfrac{\partial\phi_i(0)}{\partial z_k }\right)_{i \in I, \ k \in P_I} $is not invertible, we can assume that its rank is equal to $r \in \I{j-1}$ and that its $r-$first rows are linearly independent. Then, $(L_1,\dots,L_r)$ are linearly dependent and we can write, for all $k = r+1,\dots,j, \ L_k(\theta) = \dss_{p=1}^ra_pL_p(\theta)$. Thus, we have 
$$ \phi_k(\theta)= \left\{ \begin{array}{rlc} 
 1+ L_k(\theta) + O(\tu^2+ \dots + \theta_j^2)  & \text{if } k = 1,\dots, r \\ 
1 + \dss_{p=1}^ra_pL_p(\theta) + O(\tu^2+\dots+\theta_j^2)  & \text{if } k =r+1,\dots,j.

\end{array} 
\right.
$$
Doing linear changes of variables, we may assume that for all $k= 1 ,\dots,r$, we have $L_k(\theta) = \theta_k.$ Hence,
$$ \phi_k(\theta)= \left\{ \begin{array}{rlc} 
 1+ \theta_k + O(\tu^2+ \dots + \theta_j^2)  & \text{if } k = 1,\dots, r \\ 
1 + \dss_{p=1}^rb_p\theta_p + O(\tu^2+\dots+\theta_j^2)  & \text{if } k =r+1,\dots,j.

\end{array} 
\right.
$$

\noindent Let $\delta \in (0,1). $ Then
$$\lbrace \theta: \m{\theta_k} <\delta, \ k = 1,\dots, r  \text{ and } \m{\tk } \leq \delta^{1/2},\ k =r+1,\dots,j \} \subset  \{ \theta \in [-\pi,\pi)^d : \m{\phi_k(\theta) -1} \lesssim \delta,\ k\in I \}.$$
Moreover, 
$$\lambda_d(\lbrace \theta : \m{\theta_k} <\delta, \ k = 1,\dots, r  \text{ and } \m{\tk } \leq \delta^{1/2}, \ k =r+1,\dots,j\} ) \gtrsim \delta^{(r+j-1)/2} = \delta^{\frac{r+j-1}{2j}\cdot j}.$$
Applying Corollary $2.13$ of \cite{BD26}, we get that $C_\phi$ is  not bounded on  any $\Abd. $ 
\end{proof}

The proof shows that conditions a) and b) are necessary for $C_\phi$ to be bounded on any $\Abd$ and that they are sufficient for $\beta \geq d-3$. Now, we are going to show that $\beta = d-3$ is optimal. 
Let $\varphi\in \mathcal{O}(\DD^d,\DD) \cap \mathcal{C}^1(\ovD)$ be define by $\varphi(z_1,\dots,z_d) = z_1\cdots z_d$. Consider $$\phi : \begin{array}{ccc}
\DD^d & \rightarrow & \DD^d\\
(z_1,\dots,z_d) & \mapsto & (\varphi(z_1,\dots,z_d),\dots,\varphi(z_1,\dots,z_d),0).
\end{array}$$
Let $z \in \DD^d$. We have $\m{\varphi(z)} = 1 \Leftrightarrow \m{z_i} = 1, \forall i = 1,\dots,d$ thus $\phi_j(\xi) \in \TT$ only if $\xi\in \TT^d$. Hence, there are no $\xi$ as in condition b) for $j=1,\dots,d-1$. As $\phi_d = 0$, we have $\phi(\TT^d)\cap \TT^d = \emptyset$ and  conditon b) is also satisfied for $j= d$. From $\m{\varphi(z)} = 1 \Leftrightarrow \m{z_i} = 1, \forall i = 1,\dots,d$, we also deduce that condition a) is satisfied. Thus the function $\phi$ satifies the condition of Theorem \ref{thm_cont_geom}. Moreover by Corollary $6$ of \cite{Ko22}, we know that $C_\phi$ is bounded on $\Abd$ is and only if $\beta \geq d-3$. Therefore, $\beta = d-3$ is optimal. \\

From this theorem we can deduce a sufficient condition of boundedness which is the one we need to prove our result:
\begin{corollary}\label{condition_suff_cont}
Let $d\geq 2.$ Let $\phi\in \mathcal{O}(\DD^d,\DD^d)\cap \mathcal{C}^1(\ovD)$. Assume that
\begin{enumerate}[(a)]
    \item For all $\xi \in \TT^d$ such that $\phi(\xi) \in \TT^d, $ the derivative $d\phi(\xi)$ is invertible,\\ 
    \item For all $j \in \I{d-1},\ \phi(\ovD \setminus T_{j+1,d}) \cap T_{j,d} = \emptyset.$ 
\end{enumerate}
Then $C_\phi : A^2_{d-3}(\DD^d) \rightarrow A^2_{d-3}(\DD^d)$  is bounded.
\end{corollary} 

\begin{proof}
We want to check that these assumptions  implies the assumptions of Theorem \ref{thm_cont_geom}.\\
The condition a) is the same as condition b) of Theorem \ref{thm_cont_geom} for $j=d$. Besides, as for all $j=1,\dots,d-1$,  $T_{j+1,d}\subset T_{j,d}$ condition b) implies condition a) of Theorem \ref{thm_cont_geom}. Finally,  conditon b) implies that for $j=1,\dots,d-1$ there cannot exists $\xi \in T_{j,d}\setminus T_{j+1,d}$ such that $\phi_I(\xi) \in \TTI$ for some $I \subset \I{d}$ with $\m{I} = j$, thus condition b) of Theorem \ref{thm_cont_geom} is satisfied.  Thus, our assumptions implies the ones of Theorem \ref{thm_cont_geom} so applying it, we  deduce that $C_\phi : A^2_{d-3}(\DD^d)  \rightarrow A^2_{d-3}(\DD^d)$ is bounded.

\end{proof}

Before stating our compactness result, let us link the cases of dimension $d=2$ and $d=3$ to Theorem $8$ and $10$ of \cite{Ko22}.\\
For $d=2$ : 
\begin{align*}
 \forall \xi \in \TT^2, \phi(\xi)\in \TT^2, d\phi(\xi) \text{ is invertible} \Leftrightarrow  & \  \text{a)} \ \phi(\ovdd \setminus \TT^2) \cap \TT^2 = \emptyset\\
& \ \text{b)} \ \forall \xi \in \TT^2, \phi(\xi)\in \TT^2, d\phi(\xi) \text{ is invertible} \\
& \ \text{c)} \ \forall \xi \in T_{1,2} \setminus \TT^2, \phi_1(\xi) \in \TT^2, \dfrac{\partial \phi_1(\xi)}{\partial z_k}\neq 0, k \in P_1. 
\end{align*}
\begin{proof}
The right-hand side clearly implies the left-hand one. Now assume that $\forall \xi \in \TT^2$ such that $ \phi(\xi)\in \TT^2$ then $d\phi(\xi) \text{ is invertible} $. Then condition b) is satisfied. Condition c) is also satisfied by Julia-Caratheodory theorem as $k \in P_1$. Finally, assume that a) is not satisfied. There exists $\xi \in T_{1,2}$, for example $\xi \in \TT\times \DD$, such that $\phi(\xi) \in \TT^2$. Then, $z_2 \mapsto \phi_1(\xi_1,z_2)$ is holomorphic on $\DD$ and $\phi_2(\xi_2) \in \TT$, by the maximum principle, $\phi_1(\xi_1, \cdot)$ is constant. The same goes for $\phi_2(\xi_1, \cdot)$. Thus $\phi(\xi_1, z_2) = \phi(\xi_1,\xi_2) \in \TT^2$ for all $z_2 \in \TT$. By assumption, $d\phi(\xi_1,z_2)$ is invertible, however, $\dfrac{\partial \phi_1}{\partial z_2 }   = \dfrac{\partial \phi_2 }{\partial z_2 }  = 0 $ which is a contradiction so a) is satisfied.
\end{proof}

For $d=3$ : 
\begin{enumerate}[(1)]
  \item For any $\xi\in\TT^3$ such that $\phi(\xi)\in\TT^3$, $d\phi(\xi)$ is invertible 
  \item for any $1\leq i_1<i_2\leq 3$ and any $\xi\in\TT^3$ such that $(\phi_{i_1}(\xi),\phi_{i_2}(\xi))\in\TT^2$, either
  \begin{enumerate}[(a)]
    \item $\nabla\phi_{i_1}(\xi),\nabla\phi_{i_2}(\xi)$ are linearly independent, or
    \item $\frac{\partial \phi_{i_k}}{\partial z_j}(\xi)\neq 0$ for $j=1,2,3$ and $k=1,2.$
  \end{enumerate}
\end{enumerate}
is equivalent to 
\begin{enumerate}[(a)]
  \item $\phi(\overline{\mathbb{D}}^3 \setminus \TT^3) \cap \TT^3 = \emptyset$
  \item $\phi(\overline{\mathbb{D}}^3 \setminus T_{2,3}) \cap T_{2,3} = \emptyset$
  \item For any $\xi\in\TT^3$ such that $\phi(\xi)\in\TT^3$, $d\phi(\xi)$ is invertible 
  \item For any $\xi\in T_{2,3} \setminus \TT^3$ and any $I \subset \{1,2,3\}$ with $|I|=2$ such that $\phi_I(\xi)\in\TT^2$, $\left(\frac{\partial\phi_i(\xi)}{\partial z_k}\right)_{i \in I,\, k \in P_I}$ is invertible 
  \item For any $\xi\in T_{1,3} \setminus T_{2,3}$ and any $I \subset \{1,2,3\}$ with $|I|=1$ such that $\phi_I(\xi)\in\TT$, $\left(\frac{\partial\phi_i(\xi)}{\partial z_k}\right)_{i \in I,\, k \in P_I}$ is invertible.
\end{enumerate}

\begin{proof}
Assume $(a),(b),(c), (d), (e)$. Then $(1)$ is satisfied. Now, let $\xi \in \ovdt$  and $I \subset \{1,2,3\} $ with $\m{I }  = 2$ be such that $\phi_I(\xi )\in \TT^2$. There are three cases to consider. First, $\m{P_I} = 1 : $ impossible otherwise $\phi(\ovdt \setminus T_{2,3})\cap T_{2,3}  \neq \emptyset.$ Next, $\m{P_I}  = 2$. By $(d)$, $\left(\frac{\partial\phi_i(\xi)}{\partial z_k}\right)_{i \in I,\, k \in P_I}$ is invertible  thus $\nabla\phi_{i_1}(\xi),\nabla\phi_{i_2}(\xi)$ are linearly independent for $i_1,i_2 \in I$. Finally, if $\m{P_I} = 3$. Either $\nabla\phi_{i_1}(\xi),\nabla\phi_{i_2}(\xi)$ are linearly independent or not. If not, let us assume, for example, $\dfrac{\partial\phi_{i_1}(\xi)}{\partial z_1} = 0$. As the gradients are not linearly independent, there exists $C>0$ such that $ \dfrac{\partial\phi_{i_2}(\xi)}{\partial z_1} = 0 $ thus neither $\phi_{i_1}$ nor $\phi_{i_2}$ depend on the variable $z_2$ and $\m{P_I} \leq 2 $ which is a contradiction. Thus $\dfrac{\partial\phi_{i_j}(\xi)}{\partial z_k} \neq 0 $ for $j=1,2$ and $k=1,2,3$. Thus $(1) $ and $(2)$ are satified. \\
Now assume $(1)$ and $(2)$. Then, $(c)$ is satisfied. $(e)$ is a consequence of Julia-Caratheodory theorem's and the fact that $k \in P_I$. Then, if $(a) $ is not satisfied, there exists $\xi \in T_{2,3} $ such that $\phi(\xi) \in \TT^3$. For example, $\xi \in \TT^2 \times \DD$. The application  $z_3 \mapsto \phi_1(\xi_1,\xi_2,z_3)$ is holomorphic and its modulus reaches  its maximum in $\xi_3 \in \DD$ so by the maximum principle it is constant. The same goes for $\phi_2(\xi_1,\xi_2,\cdot)$ and $\phi_3(\xi_1,\xi_2,\cdot)$. By $(1)$, $d\phi(\xi_1,\xi_2,z_3)$ is invertible, however $\dfrac{\partial \phi_1}{\partial z_3 } = \dfrac{\partial \phi_2}{\partial z_3 } =\dfrac{\partial \phi_3}{\partial z_3 } = 0$ and that is a contradiction. Thus $(a)$ is satisfied. We show that $(b)$ is satisfied in a similar way. Now, let $\xi \in  T_{2,3} \setminus \TT^3$ and $I = \{i_1,i_2\}\subset \{1,2,3\}$ with $\m{I}  =2$ be such that $\phi_I(\xi) \in \TT^2$. If $\nabla \phi_{i_1}(\xi)$ and $\nabla\phi_{i_2}(\xi)$ are linearly independent then  $\left(\frac{\partial\phi_i(\xi)}{\partial z_k}\right)_{i \in I,\, k \in P_I}$ is invertible and if $\frac{\partial\phi_i(\xi)}{\partial z_k} \neq 0$ for $i\in I$ and $k = 1,2 ,3$ then $\m{P_I}  = 3 $ and by lemma \ref{max_principle}, $\xi \in \TT^3$ which is excluded. Finally, $(a), (b), (c), (d) $ and $(e)$ are satisfied.  
\end{proof}

 Let us finally state our compactness criterion: 
\begin{theorem}\label{thm_comp_geom}
Let $d \geq 2$. Let $\phi \in \mathcal{O}(\DD^d ,\DD^d) \cap \mathcal{C}^1(\ovD). $ Let $\beta >d-3.$ Then 
$C_\phi : \Abd \rightarrow \Abd$ \text{ is compact  if and only if } $\forall j \in \I{d},\ \phi(\ovD\setminus T_{j+1,d} ) \cap T_{j,d} = \emptyset. $ 

\end{theorem}

\begin{proof}
Let us start by  proving that the conditions are necessary. We will prove it by contrapositive. 
Assume that $\phi(\ovD \setminus T_{j+1,d}) \cap T_{j,d} \neq \emptyset$ for some $j \in \I{d}$.  There exist $\xi \in \ovD \setminus T_{j+1,d}$  and $I \subset \I{d} $ with $\m{I} = j $ such that $\phi_I(\xi) \in \TTI. $ Then $\m{P_I } \leq j. $ Indeed, if $\m{P_I}  >j $ then $\xi$ must belong to $T_{j+1,d}$ by Lemma \ref{max_principle} and that is a contradiction. Without loss of generality, we can assume that $\phi_I(\xi )  = (1,\dots,1)$. For $k \in I$, we can write 
$$\phi_k(\theta) = 1 +i\dss_{j\in P_k}a_{j,k}\tk  + O(\tu^2+\dots +\theta_d^2).$$
Let $\delta\in (0,1)$. Then, for all $k \in I$, if we chose $\m{\theta_i} \leq \delta$ for $i \in P_k$, we have $\m{\phi_k(\theta) -  1 } \lesssim \delta$. Thus,
$$ \{ \theta \in [-\pi,\pi)^d : \m{\theta_i} \leq \delta, \ i \in P_I\} \subset \{ \theta  \in[-\pi,\pi)^d : \m{\phi_j(\theta) - 1 } <\delta,\ j \in I\} $$ and 
$\lambda_d(\{  \theta \in [-\pi,\pi)^d : \m{\theta_i} \leq \delta, \ i \in P_I\})\gtrsim \delta^{\m{P_I}} = \delta^{\m{P_I}/\m{I}\cdot\m{I}}$. Applying Corollary \ref{coro_noncom}, we get that $C_\phi$ is not compact on any $\Abd. $\\

Now, let us show that the conditions  are sufficient. We want to apply Theorem \ref{result_gen_comp}. Let $\xi \in \ovD. $ Let  $I \subset \I{d}$ be such that $\phi_I(\xi) \in \TTI$ and $I$ is maximal with respect to this property. Denote by $p$ the cardinal of $ I $ and observe that $p \in\I{d-1} $ since $\phi(\TT^d) \cap \TT^d = \emptyset$. 
As its assumptions are satisfied, applying Theorem \ref{condition_suff_cont}, we get that $C_\phi : A^2_{d-3}(\DD^d) \rightarrow A^2_{d-3}(\DD^d)$  is bounded thus by Theorem $2.10$ of \cite{BD26}, there exists $C >0$ such  that for all $\eta \in \TTI $ and all $\ovd \in (0,1)^{\m{I}}, $
$$ \lambda_d(\lbrace \theta \in [-\pi,\pi)^d : \m{\phi_j(\theta) - \eta_j } < \delta_j, \ j\in I \rbrace ) \leq  C \dfrac{\dsp_{j\in I} \delta^{d-1}}{\dsp_{k=1}^d\wik^{d-2}} \leq C \veps(\ovd)\dfrac{\dsp_{j\in I} \delta^{2+\beta}}{\dsp_{k=1}^d\wik^{1+\beta }},$$ 
where $\veps(\ovd)  = \left( \dfrac{\dsp_{k=1}^d\wik}{\dsp_{j\in I} \delta_j}\right)^{\beta-d+3}. $ Let us show that $\veps(\ovd) \xrightarrow[\min\limits_{j\in I}\delta_j \to 0 ]{}0$. Indeed, applying Corollary \ref{lemme_comp}, we get that $\dfrac{\dsp_{k=1}^d\wik}{\dsp_{j\in I} \delta_j} \leq  \min\limits_{j\in I} \delta_j.  $ Thus, as $\beta > d-3 , $ we have $\veps(\ovd) \leq (\min\limits_{j\in I} \delta_j)^{\beta-d+3} \xrightarrow[\min\limits_{j\in I}\delta_j \to 0 ]{}0. $ 
 Therefore, we proved that for all $I \subset \I{d}$ such that $\phi_I$ touches the polycircle and $I $ is maximal with respect to this property, for  all $\veps >0, $ there exists $\gamma>0$ such that for all $\eta \in \TTI$ and all $\ovd \in (0,1)^{\m{I}}$  such that $\min\limits_{j\in I}\delta_j\leq  \gamma$, we have 
 $$ \lambda_d(\lbrace \theta \in [-\pi,\pi)^d : \m{\phi_j(\theta) -\eta_j } <\delta_j, \ j\in I \rbrace)\leq \veps \dfrac{\dsp_{j\in I }\delta_j^{\beta+2 } }{\dsp_{k =1}^d\wik^{1+\beta} }. $$
 Applying Theorem \ref{result_gen_comp}, we get that $
 C_\phi : \Abd \rightarrow \Abd $ is compact.

\end{proof}

For this theorem too, $\beta = d-3$ is optimal. Indeed consider as before 
$$\phi : \begin{array}{ccc}
\DD^d & \rightarrow & \DD^d\\
(z_1,\dots,z_d) & \mapsto & (\varphi(z_1,\dots,z_d),\dots,\varphi(z_1,\dots,z_d),0) 
\end{array}$$
with $\varphi(z_1,\dots,z_d) = z_1\cdots z_d$.
As $\phi_d = 0$, we have $\phi(\TT^d) \cap \TT^d = \emptyset$ and as $\m{\varphi(z) } =1 \Leftrightarrow \m{z_1} = \dots = \m{z_d}= 1$, we deduce that $\phi(\ovD \setminus T_{j+1,d} )\cap T_{j,d}= \emptyset$ for all $j = 1, \dots,d-1$. Thus, applying Theorem \ref{thm_comp_geom}, we get that $C_\phi$ is compact on any $\Abd$ with $\beta >d-3$.\\
Now, for $\xi  = (1,\dots,1)$ and $I = \I{d-1}$, we have $\phi_I(\xi ) \in \TTI$.  Let $\eio \in \TT^d$. Doing the linear change of variable $u_1 = \tu+\dots+\theta_d,$ $ u_2 = \theta_2,\dots, u_d = \theta_d$, we can write, for $k = 1,\dots, d-1, $ $\phi_k(u) = 1+iu_1+O(u_1^2)$. Let $\delta \in (0,1)$.
 We have 
$$ \{u \in [-\pi , \pi)^d : \m{u_1} \leq \delta \} \subset \{ u \in [-\pi,\pi)^d : \m{\phi_j(u) -1 } \leq \delta,\ j\in I \}$$ and 
$ \lambda_d(\{u \in [-\pi , \pi)^d : \m{u_1} \leq \delta \})  \gtrsim \delta = \delta^{1/(d-1)\cdot(d-1)}.$ Applying Corollary \ref{coro_compacite}, we get that $C_\phi$ is not compact on $A_{d-3}^2(\DD^d)$. \\

 The case with $d=2$ is an expansion of  Theorem $1$ of \cite{CCS25} to any $\beta >-1  $ under the additional assumption that $\phi \in \mathcal{C}^1(\ovdd)$:

\begin{corollary}
Let $\beta >-1$. Let $\phi \in \mathcal{O}(\DD^2,\DD^2) \cap  \mathcal{C}^1(\ovdd). $  Then $C_\phi : A^2_\beta(\DD^2) \rightarrow  A^2_\beta(\DD^2) $ is compact if and only if $\phi(\TT^2) \cap \TT^2  = \emptyset$ and $\phi(\ovdd \setminus \TT^2) \cap T_{1,2}= \emptyset. $
\end{corollary}

\vspace{5mm}


\section{Examples}
Let us give examples illustrating our compactness results: \\

\textbf{Example 1 :} Let  $\varphi :  \mathbb{D} \rightarrow \mathbb{D} $ be defined by $\varphi(z) = \dfrac{3+6z-z^2}{8} = 1+ \dfrac{z-1}{2}-\dfrac{(z-1)^2}{8}.$ We then consider \begin{center}
\begin{tabular}{ccccc}
&$\phi : $& $\DD^3$& $\rightarrow$& $\DD^3$ \\
&&$(z_1,z_2,z_3)$&$\mapsto$ & $(\varphi(z_1)\varphi(z_2)\varphi(z_3),\varphi(z_1)\varphi(z_2)\varphi(z_3),0).$
\end{tabular}
\end{center}
As in \cite{BD26}, one can prove that $C_\phi : \Ab \rightarrow \Ab $ is compact if and only if $\beta >\dfrac{-1}{2}. $

\vspace{5mm}

\textbf{Example 2  : }For $\veps\geq 0,$  we define the map $F_\veps$ on $\CC$ by
$$F_\veps(z)=\frac{3+6z-z^2}8+2i\veps(z-1)^2-i\veps(z-1)^3.$$
It is shown in \cite{Baytridisc} that there exists $\veps_1>0$ such that $F_\veps\in\mathcal O(\DD,\DD)$ provided $\veps\in[0,\veps_1)$ and $\m{F_\veps(z)} = 1 $ if and only if $z=1$. We then consider
\begin{center}
\begin{tabular}{ccccc}
&$\phi : $& $\DD^3$& $\rightarrow$& $\DD^3$ \\
&&$(z_1,z_2,z_3)$&$\mapsto$ & $(F_0(z_1)F_b(z_2)F_0(z_3),F_0(z_1)F_b(z_2)F_c(z_3),0)$

\end{tabular}
\end{center}
with $b> 0$ and  $c> 0$ small enough. One can show that $C_\phi : \Ab \rightarrow \Ab $ is compact if and only if $\beta > \dfrac{-3}{4 }. $
\smallskip

\textbf{Example 3 : } Let  $\psi :  \mathbb{D} \rightarrow \mathbb{D} $ be defined by $\psi(z) = 1+\dfrac{z-1}{2}-\dfrac{(z-1)^2}{8}+\dfrac{3}{128}(z-1)^3.$ We then consider \begin{center}
\begin{tabular}{ccccc}
&$\phi : $& $\DD^3$& $\rightarrow$& $\DD^3$ \\
&&$(z_1,z_2,z_3)$&$\mapsto$ & $(z_1z_2\psi(z_3),z_1z_2\psi(z_3),0)$
\end{tabular}
\end{center}
Applying Theorem \ref{thm:examples}, we get that $C_\phi : \Ab \rightarrow \Ab  $ is compact if and only if $\beta > \dfrac{ -1}{6 }.$
\smallskip

\textbf{Example 4 : } Consider 
$$ \phi :  \begin{tabular}{cccc}
& $\DD^4$& $\rightarrow$& $\DD^4$ \\
&$(z_1,z_2,z_3,z_4)$&$\mapsto$ & $\left(\dfrac{1+z_1z_2}{2},\dfrac{z_1+z_2+z_3+z_4}{4},\dfrac{1}{2}\left(z_1z_2+\dfrac{1+z_3}{2}\right),0\right)$
\end{tabular}
$$
Applying Theorem \ref{thm_comp_geom}, we get that $C_ \phi : A_2^2(\DD^4) \rightarrow A_2^2(\DD^4) $ is compact.

\vspace{5mm}
\vspace{5mm}

\bibliographystyle{amsplain} 
\bibliography{ref}

\end{document}